\documentclass[11pt]{article}
\usepackage{amsmath,amsfonts,amsthm,amssymb}
\usepackage{geometry}
\usepackage{enumitem}
\usepackage{kbordermatrix}
\usepackage{physics}
\usepackage[dvipdfmx]{graphicx}
\usepackage{tikz}
\usetikzlibrary{graphs,graphs.standard}
\usetikzlibrary{positioning,automata}
\usepackage{blkarray}
\usepackage{amsthm}
\theoremstyle{definition}
\newtheorem{defi}{Definition}[section]
\newtheorem{theorem}[defi]{Theorem}

\newtheorem{lemma}[defi]{Lemma}
\newtheorem{prop}[defi]{Proposition}

\newtheorem{claim}[defi]{Claim}
\newtheorem{eg}[defi]{Example}
\usepackage{bm}

\newcommand{\cl}{\mathrm{cl}}

\newcommand{\bx}{\bm x}
\newcommand{\by}{\bm y}
\newcommand{\bz}{\bm z}
\newcommand{\bp}{\bm p}
\newcommand{\bq}{\bm q}

\newcommand{\bd}{\bm d}
\newcommand{\bn}{\bm n}

\usepackage[colorlinks]{hyperref}
\hypersetup{
  colorlinks=true,
  linkcolor=blue,
  citecolor=blue,      
}

\title{Sample Complexity of Low-rank Tensor Recovery from Uniformly Random Entries}
\author{Hiroki Hamaguchi and Shin-ichi Tanigawa\thanks{Department of Mathematical Informatics, Graduate School of Information Science and Technology, University of Tokyo, 7-3-1 Hongo, Bunkyo-ku, 113-8656,  Tokyo Japan. email: {\tt hiroki4096@gmail.com, tanigawa@mist.i.u-tokyo.ac.jp}}}

\begin{document}
\maketitle

\begin{abstract}
We show that a generic tensor $T\in \mathbb{F}^{n\times n\times \dots\times n}$ of order $k$ and CP rank $d$ can be uniquely recovered from $n\log n+dn\log \log n +o(n\log \log n) $ uniformly random entries with high probability if $d$ and $k$ are constant and $\mathbb{F}\in \{\mathbb{R},\mathbb{C}\}$.
The bound is tight up to the coefficient of the second leading term
and improves on the existing $O(n^{\frac{k}{2}}{\rm polylog}(n))$ upper bound for order $k$ tensors.
The bound is obtained by showing that the projection of the Segre variety to a random axis-parallel linear subspace preserves $d$-identifiability with high probability if the dimension of the subspace is $n\log n+dn\log \log n +o(n\log \log n) $ and $n$ is sufficiently large.
\end{abstract}

\section{Introduction}
Let $\mathbb{F}\in \{\mathbb{C},\mathbb{R}\}$.
We denote by $\mathbb{F}^{n_1\times \dots \times n_k}$ the set of all order $k$ tensors of dimension $(n_1, n_2,\dots, n_k)$.
A tensor $T\in \mathbb{F}^{n_1\times \dots \times n_k}$ can be written as 
\begin{equation}\label{eq:tensor}
T =\sum_{i=1}^d x_i^1\otimes x_i^2\otimes \dots \otimes x_i^k
\end{equation}
for some positive integer $d$ and some vectors $x_i^j \in \mathbb{F}^{n_j}$ for $j=1,\dots, k$.
The smallest possible $d$ for which we can write $T$ in the form of Equation (\ref{eq:tensor}) is called the {\it CP rank}, or simply {\it rank}, of $T$. 
Analyzing low-rank tensors is a classical and still active topic in multilinear algebra. A representative mathematical motivation is a Waring-type question, which asks about the existence of special forms of given polynomials, see, e.g.,~\cite{Lansberg,bernardi}. 
In recent years, low-rank tensors have gathered significant attention from the theoretical computer science and machine learning community. This paper addresses a fundamental question in that context: unique recovery of a tensor from a few entries.

To understand the background of this application, let us first look at the low-rank matrix completion problem. In this problem, we are given a matrix with missing entries and are asked to fill in the missing entries under the assumption that the matrix is of low-rank. The problem was popularized by the Netflix Prize, which aimed to develop a better estimation method for an evaluation score matrix between users and movies to improve their recommendation system. There has been a huge amount of work in this area, both in theory and applications.
A representative theoretical result is the sampling complexity for the unique recovery of low-rank matrices: a matrix $M$ of size $n_1\times n_2$ and rank $d$ can be uniquely recovered from $O(d(n_1+n_2)\log (n_1+n_2))$ random entries if $M$ satisfies so-called the incoherence condition. 
See, e.g., \cite{keshavan2010matrix}.

Following the successful development of matrix completion techniques, there has been an attempt to extend the theory to low-rank tensors, as tensors naturally appear in multidimensional data analysis and have the advantage of having unique decompositions. Although extending ``matrices" to ``tensors" is conceptually straightforward, the theoretical analysis has proven challenging in several aspects. The most intriguing open problem is to determine the tight sampling complexity for low-rank tensor recovery.

\subsection{Problem setting}
In the literature, there are a few different settings or criteria when evaluating sampling complexity in the tensor completion problem~\cite{ghadermarzy,jain2014provable,potechin2017exact,yuan2016tensor,barak,liu2020}. Let us first clarify our setting.
\begin{description}
\item[Information theoretic v.s.~computation complexity:]
An information theoretic bound evaluates the number of observations needed to uniquely determine a tensor, whereas a computational complexity bound evaluates the number of observations that any polynomial-time algorithm requires as input to guarantee the recovery of a tensor.

For the matrix completion, those two bounds are not different;
they are both $O(n\log n)$ when $d$ is constant.
However, it is not known whether this property extends to higher order tensors. In fact, Barak and Moitra~\cite{barak} suggest that the two bounds might be different for tensors of order three.
The result of this paper is information theoretic and our argument is not algorithmic. 


\item[Exact v.s.~approximate:]
Several previous results for the tensor completion problem do not provide exact solutions
because they sometimes bound only a ``normalized" distance between the output of their algorithms and the target tensors.
The sampling complexity for exact recovery can be very large even if a normalized distance is bounded.
For example, Ghadermarzy et al.~\cite{ghadermarzy} claimed that $O(n)$ observations of random entries are enough for recovering low-rank tensors, but this bound is impossible for exact recovery since there are always an infinite number of distinct completions if the number of random observations is below $n\log n$.

The difference between exact and approximate recovery has been pointed out by Barak and Moitra~\cite{barak}. See also \cite{potechin2017exact,liu2020}. Our paper deals with exact recovery.

\item[Assumptions on entries:] 
There are pessimistic instances which are impossible to recover
using a non-trivial number of observations.
Hence, a common approach since the time of matrix completion has been to impose a regularity assumption on the values of entries.
The most standard assumption is the so-called incoherence, which assumes that the entries of each vector $x_i^j$ in the decomposition (\ref{eq:tensor}) are almost uniform.
On the other hand, we establish our result only using a {\em genericity} assumption in the sense that the entries of $x_i^j$ in the decomposition (\ref{eq:tensor}) are algebraically independent over $\mathbb{Q}$.
The set of generic instances forms a dense subset in the space of rank $d$ tensors whose complement has Lebesgue measure zero.
So a uniformly random instance with bounded entries would be generic with probability one.
However, genericity and incoherence are generally incomparable.
\end{description}

Also there are papers which deal with different problem settings:
for instance, Gaussian or generic linear observations (rather than entry-wise observations) are allowed, the low-rank assumption is imposed for other kind of rank such as Tucker rank, or sampling can be done adaptively or non-uniformly, see, e.g.,~\cite{mu2014square,karnik2024}.

\subsection{Main result}
We say that a tensor $T$ of rank $d$ is {\em generic}
if $T$ admits a decomposition (\ref{eq:tensor}) such that 
the set of entries of $x_i^j$ for $i=1,\dots, d$ and $j=1,\dots, k$ is algebraically independent over $\mathbb{Q}$.
Our main result is the following.
\begin{theorem}\label{thm:main_tensor_completion}
Let $\mathbb{F}\in \{\mathbb{R},\mathbb{C}\}$,
$d$ and $k$ be integers with $k\geq 3$,
and $T\in \mathbb{F}^{n\times n\times \dots \times n}$ be any generic tensor of order $k$ and rank $d$.
Then the probability that 
$n\log n+dn\log \log n +o(\log\log n)$ uniformly random observations
of the entries of $T$ uniquely determines $T$
tends to $1$ as $n\rightarrow \infty$.
\end{theorem}

For the unique recovery of a rank $d$ tensor, 
at least $d$ entries must be observed in each slice of the tensor.
Hence $n\log n+(d-1)n\log \log n$ is a lower bound for sampling complexity (cf.~Proposition~\ref{prop:min_degree}).
Thus, the bound of Theorem~\ref{thm:main_tensor_completion}
is tight up to the coefficient of the second leading term
and it improves on the existing $O(n^{\frac{k}{2}}{\rm polylog}(n))$ upper bound for order $k$ tensors in \cite{jain2014provable,potechin2017exact,yuan2016tensor,liu2020} when the rank $d$ is constant.

Currently no polynomial-time algorithm is known 
for recovering low-rank tensors with $O(n\log n)$ samples, even experimentally.
So our information theoretic bound supports the conjecture of Moitra and Barak~\cite{barak} about the gap between information theoretic bound and computational complexity bound.



\subsection{Outline of the proof}
There are three key ingredients in the proof of Theorem~\ref{thm:main_tensor_completion}:
\begin{description}
\item[(i)] the graph rigidity formulation of low-rank tensor completions due to Cruickshank et al.~\cite{cruickshank2023identifiability};
\item[(ii)] the sharp rigidity threshold of Erd{\"o}s--Reyni random graphs due to Lew et al.~\cite{lew2023sharp};
\item[(iii)] the sufficient conditions for the $d$-identifiability of algebraic varieties due to Chiantini and Ottaviani~\cite{CO2012} or Masaratti and Mella~\cite{MMident}.
\end{description}
Let us quickly explain how these ingredients are integrated into the proof.
Let $V_i$ be a set of $n_i$ elements for $1\leq i\leq k$ and $V$ be the disjoint union of all $V_i$. We can use a subset $E$ of $V_1 \times \dots \times V_k$ to represent the known entries in an instance of the tensor completion problem. In this manner, we encode the underlying combinatorics of each instance of the completion problem using a $k$-partite hypergraph $G=(V, E)$.
Then the rank $d$ tensor completion can be equivalently formulated as the following realization problem of a hypergraph: Given a $k$-partite hypergraph $G=(V,E)$ with $E\subseteq V_1 \times \dots \times V_k$ and $T_e\in \mathbb{F}$ for $e\in E$, 
find $\bp: V \rightarrow \mathbb{F}^d$ such that 
\begin{equation}\label{eq:system1_partite}
\sum_{j=1}^d \left( \prod_{i\in e}p_{i,j}\right) =T_e \qquad \text{for $e\in E$},
\end{equation}
where $p_{i,j}$ denotes the $j$-th entry of $\bp(i)$.
(Although the notation is slightly different, this is also a common formulation in machine learning context. See Section~\ref{sec:pre} for more details.)
The unique tensor recovery can be understood as the uniqueness of the realization of $G$ under the algebraic constraint (\ref{eq:system1_partite}).

The unique realizability of graphs under algebraic constraints is a central topic in graph rigidity theory, and a recent general framework by Cruickshank et al.~\cite{cruickshank2023identifiability} can capture the algebraic constraint system (\ref{eq:system1_partite}).
The advantage of this rigidity formulation is that 
one can exploit techniques from graph or matroid theory.
We shall adapt a recent breakthrough result by Lew et al.~\cite{lew2023sharp} about the rigidity of Erd{\"o}s--Reyni random graphs to the setting arose in the tensor completion problem.
This immediately gives a sharp probability threshold for guaranteeing the finiteness of completions of generic partially-filled low-rank tensors (or, equivalently, local rigidity in terms of graph rigidity theory). 

In order to convert the finiteness result (local rigidity) 
to the uniqueness result (global rigidity),
we use a connection to algebraic geometry based on observations in \cite{cruickshank2023identifiability}.
Specifically, the uniqueness follows by showing that 
the projection of the underlying Segre variety to a random axis-parallel linear subspace preserves $d$-identifiability. (The details will be explained in Section~\ref{sec:iden}.)
For this, we use sufficient conditions for the $d$-identifiability of algebraic varieties due to Chiantini and Ottaviani~\cite{CO2012} or Masaratti and Mella~\cite{MMident}.
Interestingly, for our particular algebraic variety, 
the sufficient condition of Masaratti and Mella~\cite{MMident} has a purely graph theoretical formulation.
Based on this, we shall introduce a new polymatroid on the edge set of hypergraphs whose rank maximality certifies the condition of Chiantini and Ottaviani or Masaratti and Mella, and show that the result by Lew et al.~\cite{lew2023sharp} can be also used to prove the uniqueness result (or, equivalently, the global rigidity of Erd{\"o}s--Reyni type random hypergraphs). 

The proof also implies that the projection of the Segre variety to a random axis-parallel linear subspace preserves $d$-identifiability if the dimension of the subspace is $n\log n+dn\log\log n+o(n\log\log n)$.
This should be compared with results with respect to generic projections.
Starting from the Noether normalization lemma, generic projections of algebraic varieties are well studied and they have also been considered in the setting of tensor analysis (see, e.g.,~\cite{mu2014square,breiding2021algebraic}). On the other hand, projections we considered here are not generic and combinatorics of projections is a key matter.

The paper is organized as follows. 
In Section~\ref{sec:pre}, we shall provide the details on the rigidity formulation of the tensor completion problem.
In Section~\ref{sec:iden}, we introduce tools from algebraic geometry and formulate a sufficient condition for unique tensor recovery in the language of graph theory (Theorem~\ref{thm:MM_test}).
In Section~\ref{sec:1d}, we first solve the 1-dimensional global rigidity problem, which corresponds to the rank one case in tensor completions.
In Section~\ref{sec:random}, we extend the result of Lew et al.~\cite{lew2023sharp} to random $k$-partite hypergraphs.
In Section~\ref{sec:main}, we introduce a new polymatroid which certifies the $d$-identifiability of the underlying algebraic varieties and then put everything together in Section~\ref{subsec:main_proof} to complete the proof of Theorem~\ref{thm:main_tensor_completion}.

\subsection{Notations}\label{subsec:notation}
\paragraph{Hypergraphs.} 
A {\em $k$-uniform hypergraph} is a pair $G=(V,E)$ of a finite set $V$ 
and a collection $E$ of subsets of $V$ of size $k$.
A $k$-uniform hypergraph is simply called a {\em $k$-graph}.
For a $k$-graph $G$, the vertex set and the edge set are denoted by $V(G)$ and $E(G)$, respectively.
A $k$-graph $G$ is said to be {\em $k$-partite} if $V$ can be partitioned into $k$ disjoint sets $V_1,\dots, V_k$ such that 
$|e\cap V_i|=1$ for every $e\in E(G)$ and every $i=1,\dots, k$.
In set theory, an edge in a $k$-partite $k$-graph is sometimes called a transversal.
Namely, given disjoint sets $V_1,\dots, V_k$,
a set $X$ with $X\subseteq \bigcup_{i=1}^k V_i$ is called a {\em transversal} (resp., {\em partial transversal})
if $|X\cap V_i|=1$ 
(resp., $|X\cap V_i|\leq 1$) for every $i=1,\dots, k$.

Let $k$ be a positive integer and $\bn=(n_1, n_2, \dots, n_k)\in \mathbb{N}^k$.
Let $K_{\bn}^k$ be the {\em complete $k$-partite $k$-graph} whose $i$-th vertex class consists of $n_i$ vertices.
When $\bn=(n, n, \dots, n)$, $K_{\bn}^k$ is said to be {\em balanced} and is denoted simply by $K_n^k$.

Let $G$ be a $k$-partite $k$-graph.
A vertex set $X\subseteq V(G)$ is said to be a {\em k-partite clique} in $G$ if 
the induced subgraph of $G$ by $X$ is complete.
Since we only look at $k$-partite $k$-graphs in this paper, 
a $k$-partite clique is simply called a {\em clique}.

In a hypergraph $G$, a sub-hypergraph $H$ is said to be {\em spanning}
if $V(H)=V(G)$.

\paragraph{Erd{\"o}s--Reyni type random subgraphs.} 
For $n\in \mathbb{N}$, we consider two random subgraph models of $K_{n}^k$.
The first one is obtained from $K_{n}^k$ by deleting each hyperedge with probability $1-p$ for $p\in [0,1]$.
This random graph is denoted by $G^k(n,p)$ or even simply by $G(n,p)$
since we only look at $k$-graphs in this paper.
The second one, denoted by $G^k(n,m)$ or simply by $G(n,m)$ for a nonnegative integer $m$, follows the uniform distribution over all subgraphs of $K_{n}^k$ having exactly $m$ hyperedges.

We say that a property holds {\rm a.a.s.}~if 
the probability that $G(n,p)$ (resp., $G(n,m)$) has the property 
tends to one when $n\rightarrow \infty$.
It is well known that, for a monotone property $P$, 
$G(n,p)$ and $G(n,m)$ are asymptotically equivalent
in the sense that 
$G(n,p)$ satisfies $P$ a.a.s.~if and only if 
$G(n,m)$ satisfies $P$ a.a.s.~for $m=\lceil n^k p\rceil$.
See, e.g., \cite{freeze} for more details.

A well-known fact in the random graph theory states that 
$p=\frac{\log n + (d-1)\log \log n+o(\log\log n)}{n}$ is the probability threshold for 
the (ordinary) Erd{\"o}s--Reyni random graph to have minimum degree at least $d$. It is also known that the analysis can be extended to random hypergraphs. The same analysis also gives the minimum degree threshold for $G(n,p)$. 
\begin{prop}\label{prop:min_degree}
Let $n, p\in \mathbb{N}$ and $p\in [0,1]$.
Define 
\begin{equation}\label{eq:p}
\begin{split}
p^{(d)}_+&=\frac{\log n+(d-1)\log \log n+\log\log \log n}{n^{k-1}}, \\
p^{(d)}_-&=\frac{\log n+(d-1)\log \log n -\log\log \log n}{n^{k-1}}.
\end{split}
\end{equation}

If $p\geq p^{(d)}_+$,
then a.a.s.~the minimum degree of $G(n,p)$ is at least $d$.

If $p\leq p^{(d)}_-$,
then a.a.s.~the minimum degree of $G(n,p)$ is at most $d-1$.
\end{prop}
\paragraph{Polymatroids.} 
We consider a matroid as a pair $M=(E, r)$ of a finite set $E$ and a function $r:2^E\rightarrow \mathbb{Z}$ satisfying
(R1) $r(\emptyset)=0$;
(R2) $r(e):=r(\{e\})\leq 1$ for $e\in E$; 
(R3) $r(X)\leq r(Y)$ for $X\subseteq Y\subseteq E$; and 
(R4) $r(X)+r(Y)\geq r(X\cap Y)+r(X\cup Y)$ for $X, Y\subseteq E$.

When $r$ satisfies (R1), (R3) and (R4), the pair $P=(E,r)$ is called a {\em polymatroid}.
A representative example of a polymatroid is a linear polymatroid
obtained by assigning a linear subspace $L_e$ for each $e\in E$
and setting $r(F)=\dim \langle L_e: e\in E\rangle$ for $F\subseteq E$,
where $\langle \cdot \rangle$ means the linear span.

Let $P=(E,r)$ be a polymatroid. 
The {\em rank} of the polymatroid is $r(E)$,
and $X\subseteq E$ is said to be a {\em spanning set} if $r(X)=r(E)$.
The {\em closure} $\cl_P:2^E\rightarrow 2^E$ is defined by
$\cl_P(X)=\{e\in E: r(X+e)=r(X)\}$.
$\cl_P$ is indeed a closure operator
as, by (R3) and (R4), it satisfies the closure axiom: 
(C1) $X\subseteq \cl_P(X)$ for $X\subseteq E$;
(C2) $\cl_P(X)\subseteq \cl_P(Y)$ for $X\subseteq Y\subseteq E$;
(C3) $\cl_P(\cl_P(X))=\cl_P(X)$ for any $X\subseteq E$.

\section{Rigidity Formulation of the Tensor Completion Problem}\label{sec:pre}
In this section, we introduce a graph rigidity formulation of the low-rank tensor completion problem.
The material in this section is based on \cite{cruickshank2023identifiability}.
It should be noted that our formulation is a tensor analogue of that of matrix completions due to Singer and Cucuringu~\cite{singer2010uniqueness} and Kir{\'a}ly, Theran, and Tomioka~\cite{kiraly2015algebraic}.

Let $V_1,\dots, V_k$ be disjoint finite non-empty sets with $n_i=|V_i|$.
Denote $V=V_1\cup \dots\cup V_k$ and $\bn=(n_1,\dots, n_k)$ with $N=\sum_i n_i$ and $M=\prod_i n_i$.
Recall that $K^k_{\bn}$ denotes the complete $k$-partite $k$-graph on $V$. 
In the subsequent discussion, each coordinate of $\mathbb{F}^{N}$ is indexed by each vertex of $K^k_{\bn}$
and each coordinate of $\mathbb{F}^M$ is indexed by an edge of $K^k_{\bn}$.

\subsection{Geometric interpretation of low-rank tensors}
We shall look at (the affine variant) of the {\em Segre map},
that is, a map $\sigma_{\bn}:\mathbb{F}^N\rightarrow \mathbb{F}^M$ given by
\[
\sigma_{\bn}(\bx)=\Big(x_{i_1}x_{i_2}\dots x_{i_k}\Big)_{(i_1,\dots, i_k)\in E(K^k_{\bn})}
\qquad (\bx=(x_1,\dots, x_N)\in \mathbb{F}^N)
\]
Then $\overline{{\rm im}\ \sigma_{\bn}}$ is known as the affine cone of the {\em Segre variety} of order ${\bn}=(n_1,\dots, n_k)$.
The set of tensors of rank at most $d$ can be understood
by looking at the $d$-secant of the Segre variety.
In general, for an affine variety ${\cal V}$,
its {\em $d$-secant} $S_d(\mathcal{V})$ is defined as the Zariski closure of the set of points linearly spanned by $d$ points in ${\cal V}$.

A {\em $d$-dimensional point configuration} is a tuple $\bp=(p_1,\dots, p_N)\in (\mathbb{F}^d)^N$ of $N$ points in $\mathbb{F}^d$.
The $j$-th {\em coordinate vector} of $\bp$ is the $j$-th row vector of $\bp=(p_1,\dots, p_N)$ when $\bp$ is viewed as a $d\times N$ matrix. 
Equivalently, if we denote the $j$-th entry of $p_i\in \mathbb{F}^d$ by $p_{i,j}$, the $j$-th coordinate vector is $(p_{1,j},p_{2,j},\dots, p_{N,j})\in \mathbb{F}^N$.

We define the $d$-sum $\sigma^d_{\bn}: (\mathbb{F}^d)^N\rightarrow \mathbb{F}^M$ of $\sigma_{\bn}$ to be the sum of $\sigma_{\bn}$ applied to each coordinate vector, i.e., 
\[
\sigma^d_{\bn}(\bp)=\left(\sum_{j=1}^d p_{i_1,j} p_{i_2,j}\dots p_{i_k,j}\right)_{(i_1,\dots, i_k)\in E(K^k_{\bn})} 
\qquad (\bp=(p_1,\dots, p_N)\in (\mathbb{F}^d)^N).
\]
Then 
$\overline{{\rm im}\ \sigma^d_{\bn}}$ is the $d$-secant 
of $\overline{{\rm im}\ \sigma_{\bn}}$.

The $d$-secant of a Segre variety appears as a basic mathematical object in the analysis of low-rank tensors.
To see this, consider first a rank-1 tensor $T\in \mathbb{F}^{n_1\times n_2\times \dots \times n_k}$.
Since $T$ is rank 1, $T=\bx_1\otimes \bx_2\otimes \dots \otimes \bx_k$
for some $\bx_i\in \mathbb{F}^{n_i}$ for $i=1,\dots, k$.
Let $\bx=(x_1, x_2,\dots,x_N) \in \mathbb{F}^N$ be an $N$-dimensional vector obtained by concatenating $\bx_1, \bx_2,\dots, \bx_k$.
Then the $(i_1,\dots, i_k)$-th entry of $T$ is 
$x_{i_1}x_{i_2}\dots x_{i_k}$, which is exactly equal to the $(i_1,\dots, i_k)$-th entry of $\sigma_{\bn}(\bx)$.
Hence, the image of $\sigma_{\bn}(\bx)$ coincides with the set of tensors of rank at most one.
Since $\sigma_{\bn}^d$ is the sum of $d$ copies of $\sigma_{\bn}$,
the image of $\sigma_{\bn}^d$ is the set of tensors of rank at most $d$.
This can be also verified by entry-wise comparison by 
observing that, for a tensor $T$ of rank $d$ and 
a $d$-dimensional point configuration $\bp$,
$T=\sigma_{\bn}^d(\bp)$ holds if and only if 
each coordinate vector of $\bp$ gives each term of the tensor decomposition (\ref{eq:tensor}) of $T$.

As explained above, $\sigma_{\bn}$ and $\sigma_{\bn}^d$ give parameterisations of a Segre variety and its $d$-secant, respectively.
However, this parameterisation has redundancy since $\sigma_{\bn}$ or $\sigma_{\bn}^d$ are not bijective.
This is due to the fact that the tensor decomposition~(\ref{eq:tensor}) has a degree of freedom of permutating indices $i$ and scaling each $x_{i}^j$ by $\lambda_i^j$ with $\prod_{j=1}^k \lambda_i^j=1$.
In terms of point configurations, this degree of freedom can be captured as follows.
Consider the situation where the $k$-fold product $GL(d,\mathbb{F})^k$ of the general linear group acts on 
$(\mathbb{F}^d)^N$ such that,
for $(A_1,\dots, A_k)\in GL(d,\mathbb{F})^k$ and 
 $\bp=(p_1,\dots, p_N)\in (\mathbb{F}^d)^N$,
\[
((A_1,\dots, A_k)\cdot \bp)_v=A_j p_v\quad \text{for } v\in V_j. 
\]
Namely, each $A_j$ acts on points indexed by vertices in $V_j$.
Observe that 
the value of $\sigma^d_{\bn}$ is invariant by the action of $(A_1,\dots, A_k)$ 
if $A_i=D_i\Sigma$ 
for some permutation matrix $\Sigma$ and some diagonal matrix $D_i$ satisfying $\prod_{i=1}^k D_i=I_d$.
Such $(A_1,\dots, A_k)$ is called a {\em stabilizer}.
The set of all stabilizers
forms a Lie subgroup of $GL(d,\mathbb{F})^k$ of dimension $d(k-1)$.

We say that $\bp, \bq\in \mathbb(\mathbb{F}^d)^N$ are {\em congruent} if 
there is a stabilizer that maps $\bp$ to $\bq$.
If $\bp$ and $\bq$ are congruent, then $\sigma_{\bn}^d(\bp)=\sigma_{\bn}^d(\bq)$ and $\bp, \bq$ represent different decompositions of the same rank $d$ tensor.

\begin{eg}\label{eg:congruent}
    Let us consider the case where $d=2,k=3$, $\bn=(2,2,2)$, and $V_1=\{1,2\}, V_2=\{3,4\}, V_3=\{5,6\}$
    For a point configuration $\bp=(p_1,\dots, p_6) \in (\mathbb{F}^2)^6$, the corresponding tensor $T$, in the form of\eqref{eq:tensor}, is 
    \begin{equation*}
        T= \mqty(p_{1,1} \\ p_{2,1}) \otimes \mqty(p_{3,1} \\ p_{4,1}) \otimes \mqty(p_{5,1} \\ p_{6,1}) + \mqty(p_{1,2} \\ p_{2,2}) \otimes \mqty(p_{3,2} \\ p_{4,2}) \otimes \mqty(p_{5,2} \\ p_{6,2}).
    \end{equation*}
     The map $\sigma^2(\bp)$ returns 
              \vspace{-1em}
         \begin{equation*}
        \kbordermatrix{
               &                                                 \\
        (1,3,5)&p_{1,1} p_{3,1} p_{5,1} + p_{1,2} p_{3,2} p_{5,2}\\
        (1,3,6)&p_{1,1} p_{3,1} p_{6,1} + p_{1,2} p_{3,2} p_{6,2}\\
        (1,4,5)&p_{1,1} p_{4,1} p_{5,1} + p_{1,2} p_{4,2} p_{5,2}\\
        (1,4,6)&p_{1,1} p_{4,1} p_{6,1} + p_{1,2} p_{4,2} p_{6,2}\\
        (2,3,5)&p_{2,1} p_{3,1} p_{5,1} + p_{2,2} p_{3,2} p_{5,2}\\
        (2,3,6)&p_{2,1} p_{3,1} p_{6,1} + p_{2,2} p_{3,2} p_{6,2}\\
        (2,4,5)&p_{2,1} p_{4,1} p_{5,1} + p_{2,2} p_{4,2} p_{5,2}\\
        (2,4,6)&p_{2,1} p_{4,1} p_{6,1} + p_{2,2} p_{4,2} p_{6,2}}
        \in \mathbb{F}^8,
    \end{equation*}
    whose $(i_1,i_2,i_3)$-th entry coincides with that of $T$.
    Let $a,b$ be nonzero numbers in $\mathbb{F}$,
    and suppose 
    $(A_1, A_2, A_3)\in GL(2,\mathbb{F})^3$ is given by  
    $A_1=\mqty(1&0\\0&a)\mqty(0&1\\1&0)$,
    $A_2=\mqty(1&0\\0&b)\mqty(0&1\\1&0)$,
    $A_3=\mqty(1&0\\0&\frac{1}{ab})\mqty(0&1\\1&0)$.
    Then it forms a stabilizer, and its action maps $\bp$ to 
   $\bq=(q_1,\dots, q_6)$ with 
   $q_1=A_1p_1, q_2=A_1p_2, q_3=A_2p_3, q_4=A_2p_4, q_5=A_3p_5, q_6=A_3p_6$.
   Then, $\bq$ gives rise to another decomposition of the same tensor $T$:
     \begin{equation*}
         T = \mqty(a p_{1,2} \\ a p_{2,2}) \otimes \mqty(b p_{3,2} \\ b p_{4,2}) \otimes \mqty(\frac{1}{ab}p_{5,2} \\ \frac{1}{ab}p_{6,2}) + \mqty(p_{1,1} \\ p_{2,1}) \otimes \mqty(p_{3,1} \\ p_{4,1}) \otimes \mqty(p_{5,1} \\ p_{6,1}).
     \end{equation*}
\end{eg}

\subsection{Rigidity interpretation of unique completions}
\paragraph{Local and global rigidity.}
Since each coordinate of $\mathbb{F}^M$ is indexed by an edge of $K^k_{\bn}$,
the projection $\pi_G$ of $\mathbb{F}^M$ to the linear subspace indexed by the edges of $G$ is defined for each subgraph $G$ of $K^k_{\bn}$.
The {\em $d$-dimensional rigidity map} $f_G^d$ of $G$ is then defined 
as $f_G^d=\pi_G\circ \sigma^d_{\bn}$, or more explicitly,
\[
f^d_G(\bp):=\pi_G\circ \sigma^d_{\bn}(\bp)=\left(\sum_{j=1}^d p_{i_1,j} p_{i_2,j}\dots p_{i_k,j}\right)_{(i_1,\dots, i_k)\in E(G)}
\qquad (\bp=(p_1,\dots, p_N)\in (\mathbb{F}^d)^N)
\]

A {\em $d$-dimensional framework} $(G,\bp)$ is a pair of a $k$-graph $G$ and $\bp\in (\mathbb{F}^d)^N$.
We say that a framework $(G,\bp)$ is {\em globally rigid} if every $\bq\in (\mathbb{F}^d)^N$ satisfying $f_G^d(\bp)=f_G^d(\bq)$ is congruent to $\bp$.
We say that $(G,\bp)$ is {\em locally rigid} if there is an open neighborhood $N_{\bp}$ of $\bp$ in $(\mathbb{F}^d)^N$
such that every $\bq\in N_{\bp}$ satisfying $f_G^d(\bp)=f_G^d(\bq)$ is congruent to $\bp$.

\paragraph{Unique recovery of tensors and global rigidity.} In terms of tensor completions, the projection map $\pi_G$ corresponds to the {\em masking map}. 
More specifically,  $\sigma_{\bn}^d(\bp)$ is a tensor of rank at most $d$
and $f_G^d(\bp)=\pi_G(\sigma_{\bn}^d(\bp))$ is a partially filled tensor 
whose entries indexed by $E(K_{\bn}^k)\setminus E(G)$ are missing.
Hence, for frameworks $(G,\bp), (G,\bq)$ with $f_G^d(\bp)=f_G^d(\bq)$, the corresponding tensors $\sigma^d_{\bn}(\bp)$
and $\sigma^d_{\bn}(\bq)$  have the same value on the entries indexed by the elements of $E(G)$.
So $(G,\bp)$ is globally rigid if and only if  
$\sigma^d_{\bn}(\bp)$ is the unique rank $d$ tensor which can be recovered from a partially filled tensor $\pi_G(\sigma_{\bn}^d(\bp))$.
In this way we can convert the unique recovery question of rank $d$ tensors to the problem of deciding the global rigidity of $d$-dimensional frameworks.

\begin{eg}\label{ex:2}
    Consider the same instance as Example~\ref{eg:congruent}
    and consider a $3$-graph $G=(V=\{1,2\} \cup \{3,4\} \cup \{5,6\}, E=\{(1,3,5),(1,4,6)\})$.
    The $2$-dimensional rigidity map of $G$ is written as
    \begin{equation*}
    f_G^2(\bp) =
    \kbordermatrix{
        &  \\
       (1,3,5) & p_{1,1}p_{3,1}p_{5,1}+p_{1,2}p_{3,2}p_{5,2} \\
       (1,4,6) & p_{1,1}p_{4,1}p_{6,1}+p_{1,2}p_{4,2}p_{6,2}
    }.
    \end{equation*}
    Consider another point configuration $\bp' = (p_1, p_2', p_3, p_4, p_5, p_6)$, where $p_2'\in \mathbb{F}^2$ is an arbitrary point with $p_2'\neq p_2$.
    Since $G$ has no edge containing $2$, we have 
    $f_G^d(\bq)=f_G^d(\bp')$. However, this configuration $\bp'$ is not congruent to $\bp$, which means that $(G,\bp)$ is neither locally nor globally rigid.
\end{eg}

Although we have defined the unique recovery problem in terms of graph rigidity, the formulation is equivalent to a standard formulation in machine learning context.
In machine learning, the tensor completion problem is often formulated as
the following optimization problem:
\[
\displaystyle{
\min_{\bq\in (\mathbb{R}^d)^N} 
\sum_{(i_1,\dots, i_k)\in E(G)} \left\|
T_{(i_1,\dots, i_k)} - \sum_{j=1}^d 
q_{i_1,j}q_{i_2,j}\cdots q_{i_k,j}  \right\|^2}.
\]
Let $\bp\in (\mathbb{R}^d)^N$ be such that $T=\sigma_{\bn}^d(\bp)$.
Then the global rigidity of $(G,\bp)$ is equivalent to saying that 
the above optimization problem has the unique solution (up to congruence),
and the local rigidity of $(G,\bq)$ is equivalent to
the finiteness of the optimizer set (modulo congruence).

\paragraph{Checking local rigidity.}
It is a well-known fact in rigidity theory that checking local rigidity can be reduced to a linear algebraic question if $\bp$ is assumed to be generic.
The same trick can be applied even in the current setting. 
Indeed, if $\bp$ is generic, then $\bp$ is a regular value of $f_G^d$, and hence by the inverse function theorem, 
the dimension of the fiber of $f_G^d$ at $f_G^d(\bp)$ is equal to the dimension of the kernel of the Jacobian $Jf_G^d(\bp)$ of $f_G^d$ at $\bp$.
Since $Jf_G^d$ is invariant over congruent point configurations and the set of point configurations congruent to $\bp$ forms a manifold of dimension $d(k-1)$
(as it can be identified with the stabilizer group),
we always have
\[
\rank Jf_G^d(\bp)\leq dN-d(k-1)
\]
and the equality holds if and only if 
only possible point configuration $\bq$ with $f_G^d(\bp)=f_G^d(\bq)$ is 
congruent to $\bp$ in a neighborhood of $\bp$. 
The following proposition summarizes this discussion.
\begin{prop}\label{prop:infinitesimal}
Let $G$ be a $k$-partite $k$-graph with $N$ vertices
and $(G,\bp)$ be a generic $d$-dimensional framework over $\mathbb{F}$.
Then $(G,\bp)$ is locally rigid in $\mathbb{F}^d$ if and only if 
$\rank Jf_G^d(\bp)=dN-d(k-1)$.
\end{prop}

\begin{eg}
Consider the same instance as that in Example~\ref{ex:2},
where $V(G) = \{ {1}, {2} \} \cup \{ {3}, {4} \} \cup \{ {5}, {6} \}$
and $E(G)=\{(1,3,5),(1,4,6)\}$.
The Jacobian $J f_G^2(\bp)$ of $f_G^2$ at $\bp$ is 
\begin{equation*}
        \scalebox{0.9}{$
        \kbordermatrix{\setlength{\kbcolsep}{0pt}
       &    p_{1,1}&p_{1,2}&p_{2,1}&p_{2,2}&p_{3,1}&p_{3,2}&p_{4,1}&p_{4,2}&p_{5,1}&p_{5,2}&p_{6,1}&p_{6,2}\\
       &p_{3,1} p_{5,1} & p_{3,2} p_{5,2} & 0  &  0 & p_{1,1} p_{5,1} & p_{1,2} p_{5,2} & 0  & 0  & p_{1,1} p_{3,1} & p_{1,2} p_{3,2} & 0  & 0 \\
       &p_{4,1} p_{6,1} & p_{4,2} p_{6,2} & 0  & 0  & 0  & 0  & p_{1,1} p_{6,1} & p_{1,2} p_{6,2} &  0 &  0 & p_{1,1} p_{4,1} & p_{1,2} p_{4,2}
        }$},
    \end{equation*}  
    which does not satisfy the rank criterion of Proposition~\ref{prop:infinitesimal}. 
    On the other hand,  the Jacobian $J f_{K_{\bn}^k}^2(\bp)$ 
    of $f_{K_{\bn}^k}^2$ at $\bp$ is
    \begin{equation*}
        \scalebox{0.85}{$
        \kbordermatrix{\setlength{\kbcolsep}{0pt}
       &    p_{1,1}&p_{1,2}&p_{2,1}&p_{2,2}&p_{3,1}&p_{3,2}&p_{4,1}&p_{4,2}&p_{5,1}&p_{5,2}&p_{6,1}&p_{6,2}\\
       &p_{3,1} p_{5,1} & p_{3,2} p_{5,2} & 0  & 0  & p_{1,1} p_{5,1} & p_{1,2} p_{5,2} & 0  &  0 & p_{1,1} p_{3,1} & p_{1,2} p_{3,2} & 0  & 0 \\
       &p_{3,1} p_{6,1} & p_{3,2} p_{6,2} &  0 & 0  & p_{1,1} p_{6,1} & p_{1,2} p_{6,2} &  0 & 0  & 0  &  0 & p_{1,1} p_{3,1} & p_{1,2} p_{3,2}\\
       &p_{4,1} p_{5,1} & p_{4,2} p_{5,2} & 0  & 0  &  0 &   0& p_{1,1} p_{5,1} & p_{1,2} p_{5,2} & p_{1,1} p_{4,1} & p_{1,2} p_{4,2} &  0 & 0 \\
       &p_{4,1} p_{6,1} & p_{4,2} p_{6,2} & 0  &  0 &  0 &  0 & p_{1,1} p_{6,1} & p_{1,2} p_{6,2} &  0 & 0  & p_{1,1} p_{4,1} & p_{1,2} p_{4,2}\\
       & 0 &  0 & p_{3,1} p_{5,1} & p_{3,2} p_{5,2} & p_{2,1} p_{5,1} & p_{2,2} p_{5,2} & 0  &  0 & p_{2,1} p_{3,1} & p_{2,2} p_{3,2} &  0 &0  \\
       & 0 & 0  & p_{3,1} p_{6,1} & p_{3,2} p_{6,2} & p_{2,1} p_{6,1} & p_{2,2} p_{6,2} & 0  & 0  &  0 &  0 & p_{2,1} p_{3,1} & p_{2,2} p_{3,2}\\
       & 0 & 0  & p_{4,1} p_{5,1} & p_{4,2} p_{5,2} & 0  & 0  & p_{2,1} p_{5,1} & p_{2,2} p_{5,2} & p_{2,1} p_{4,1} & p_{2,2} p_{4,2} &  0 & 0 \\
       & 0 &0   & p_{4,1} p_{6,1} & p_{4,2} p_{6,2} &  0 &  0 & p_{2,1} p_{6,1} & p_{2,2} p_{6,2} & 0  & 0  & p_{2,1} p_{4,1} & p_{2,2} p_{4,2}
        }
        $}.
    \end{equation*}    
    We can check its full rankness by randomly assigning values to the entries of $\bp$ and confirming its rank to be 8, which coincides with $dN-d(k-1)=8$.
    Thus, Proposition~\ref{prop:infinitesimal} asserts that $(K_{\bn}^k,\bp)$ is locally rigid in $\mathbb{F}^2$ if $\bp$ is generic.
\end{eg}

\paragraph{Generic Rigidity.}
We say that a $k$-partite $k$-graph $G$ is {\em globally rigid} (resp., {\em locally rigid}) in $\mathbb{F}^d$ if 
$(G,\bp)$ is globally rigid (resp., locally rigid) for all generic $d$-dimensional point configurations $\bp$ over $\mathbb{F}$.
This definition is motivated by the following fundamental concept in rigidity theory, {\em  generic rigidity}, due to Asimov and Roth~\cite{asimow}.

Proposition~\ref{prop:infinitesimal} implies that,
if $(G,\bp)$ is locally rigid for a generic $d$-dimensional point configuration $\bp$, then
$(G,\bq)$ is locally rigid for {\em all} generic point configurations $\bq$.
In this sense,  local rigidity is determined by underlying hypergraphs as long as point configurations are assumed to be generic.
Motivated by this fact, 
we have defined the local rigidity of a $k$-graph $G$  in $\mathbb{F}^d$ such that $(G,\bp)$ is locally rigid for a (or equivalently, any) generic $d$-dimensional point configuration $\bp$.

Establishing the global rigidity counterpart of this property is a subtle question. By using Chevalley's theorem, one can show that 
if $(G,\bp)$ is globally rigid for a generic $\bd$-dimensional point configuration $\bp$ over $\mathbb{C}$, then
$(G,\bq)$ is globally rigid for all generic $d$-dimensional point configurations $\bq$ over $\mathbb{C}$. See \cite{Gortler2014} for more details.
However, this may not be the case when the underlying field is $\mathbb{R}$. See 
\cite{GHT10} for a positive result in the Euclidean distance rigidity problem and \cite{JJTunique} for a negative result in the matrix completion problem.
Nevertheless, it is still convenient to say that a $k$-partite $k$-graph $G$ is globally rigid in $\mathbb{F}^d$ if 
$(G,\bp)$ is globally rigid for all generic $d$-dimensional point configurations $\bp$ over $\mathbb{F}$.

\paragraph{Unique recovery from random sampling.}
Our main theorem (Theorem~\ref{thm:main_tensor_completion}) is about the sampling complicity of the unique recovery of a rank $d$ tensor $T$ when sampling is done uniformly randomly.
In our formulation, 
this corresponds to looking at the masking map $\pi_G$
of  the Erd{\"o}s--Reyni type random subgraph $G=G(n,p)$ or $G(n,m)$ defined in Section~\ref{subsec:notation}.
Hence, the unique recovery problem of a rank $d$ tensor $T$ is equivalent to asking the global rigidity of $(G(n,p),\bp)$ or $(G(n,m),\bp)$ 
for a $d$-dimensional point configuration $\bp$ with $T=\sigma^d_{\bn}(\bp)$.
Since we further assume  that $T$ is generic in Theorem~\ref{thm:main_tensor_completion},
Theorem~\ref{thm:main_tensor_completion} follows by analyzing the probability threshold 
for the global rigidity of $G(n,p)$ or $G(n,m)$.
Thus, in the remaining of the paper, we focus on the latter problem.

\paragraph{Generic rigidity matroid.}
The language of matroid theory will be useful for further investigating the combinatorics of rigidity.
The {\em generic tensor rigidity matroid}
${\cal T}_{\bn,d}$ is defined 
as the row matroid of $J\sigma^d_{\bn}(\bp)$ at a (or equivalently, any) generic point configuration $\bp\in (\mathbb{F}^d)^N$.
Note that each row of $J\sigma^d_{\bn}(\bp)$ is indexed by an edge of $K_{\bn}^k$,
so we may consider ${\cal T}_{\bn,d}$ as a matroid on $E(K_{\bn}^k)$.
In view of Proposition~\ref{prop:infinitesimal},
$G\subset K_{\bn}^k$ is locally rigid if and only if the rank of $E(G)$ in  ${\cal T}_{\bn,d}$ is $dN-d(k-1)$.

\subsection{1-dimensional local rigidity}\label{subsec:1d_local}
As a simplest example, in this subsection we shall look at the 1-dimensional local rigidity.

Let $G$ be a $k$-partite $k$-graph.
By Proposition~\ref{prop:infinitesimal}, 
$G$ is locally rigid in $\mathbb{F}^d$ if and only if
$\rank Jf_G^d(\bp)=dN-d(k-1)$ for a generic point configuration $\bp$.
When $d=1$, a special structure of $Jf_G^d(\bp)$ can be exploited to simplify the problem.

We introduce the incidence matrix $I_G$ of $G$ to be 
the matrix of size $|V(G)|\times |E(G)|$
whose each row is indexed by a vertex, each column is indexed by an edge,
and each entry is given by
\[
I_G[v,e]=\begin{cases}
1 & (v\in e)\\
0 & (\text{otherwise})
\end{cases}\qquad (v\in V(G), e\in E(G)).
\]
When $G$ is a $k$-partite, we always have 
\begin{equation}\label{eq:IG_lower}
\dim \ker I_G^{\top}
\geq k-1.
\end{equation}
Indeed, if $V_1,\dots, V_k$ denote the $k$-partition of the vertex set $V(G)$, 
the vector $\bx_i\in \mathbb{F}^{N}$ 
defined by 
\begin{equation}\label{eq:1d_trivial}
\bx_i(v)=\begin{cases}
    1 & (v\in V_1) \\
    -1 & (v\in V_i) \\
    0 & (v\in V(G)\setminus (V_1\cup V_i) )
    \end{cases}
\end{equation}
belongs to $\ker I_G^{\top}$ for $i=2,\dots, k$.
Since $\bx_2,\dots, \bx_{k}$ are linearly independent, 
we have (\ref{eq:IG_lower}).
We call a linear combination of $\bx_2,\dots, \bx_{k}$   a {\em trivial kernel vector} of $I_G$,
and the space of all trivial kernel vectors is called the {\em trivial kernel} of $I_G$.

\begin{lemma}\label{lem:1d_local}
Let $G$ be a $k$-partite $k$-graph with $N$ vertices.
Then $G$ is locally rigid in $\mathbb{F}^1$ if and only if $\rank_{\mathbb{F}} I_G=N-(k-1)$.    
\end{lemma}
\begin{proof}
Pick a generic one-dimensional point-configuration $\bp=(p_1,p_2,\dots,p_N)\in \mathbb{F}^N$.
Then 
\[
(Jf_G^1(\bp))[e,u]=\begin{cases}
\prod_{v: v\in e\setminus\{u\}} p_v & (u\in e)\\
0 & (\text{otherwise}).
\end{cases}
\]
Hence, if we multiply each column of $u$ by $p_u$
and then divide each row of $e$ by $\prod_{v\in e} p_v$,
 $Jf_G^1(\bp)$ is converted to $I_G^{\top}$.
Thus the lemma follows.
\end{proof}

The 1-dimensional global rigidity will be discussed in Section~\ref{sec:1d}.
Characterizing higher dimensional local/global rigidity turns out to be a much harder question. See~\cite{cruickshank2023identifiability} for more details.

\begin{eg}\label{eg:1d_local}
    Let us consider the case when 
    $d=1, k=3,\bn=(2,2,2)$, $G=(V=\{1,2\}\cup\{3,4\}\cup\{5,6\}, 
    E=\{(1,3,5),(1,3,6),(1,4,6),(2,3,6),(2,4,5)\}$,
    and $\bp=(p_1,\dots, p_6)\in \mathbb{F}^6$.
    The transpose of the Jacobian of the rigidity map is 
    \begin{equation*}
    J f_G^1(\bp)^{\top} = \kbordermatrix{
       & (1,3,5)&(1,3,6)&(1,4,6)&(2,3,6)&(2,4,5) \\
    1&p_{3} p_{5} & p_{3} p_{6}  & p_{4} p_{6} &  0 &  0  &  \\
    2& 0 & 0 & 0 & p_{3} p_{6} & p_{4} p_{6} \\
    3& p_{1} p_{5}  & p_{1} p_{6} & 0 & p_{2} p_{5} & 0 \\
    4& 0 & 0 &  p_{1} p_{6} & 0 & p_{2} p_{5} \\
    5& p_{1} p_{3} & 0 & 0 & 0 & p_{2} p_{4} \\
    6& 0 & p_{1} p_{3} & p_{1} p_{4} & p_{2} p_{3} & 0 
    },
    \end{equation*}
    which can be converted to the incidence matrix 
    \begin{equation}\label{eq:IGexample}        
    I_G = \kbordermatrix{
    & (1,3,5)&(1,3,6)&(1,4,6)&(2,3,6)&(2,4,5) \\
    1&1 & 1 & 1  & 0   & 0    \\
    2&0 & 0 & 0  & 1   & 1    \\
    3&1 & 1 & 0  & 1   & 0    \\
    4&0 & 0 & 1  & 0   & 1    \\
    5&1 & 0 & 0  & 0   & 1    \\
    6&0 & 1 & 1  & 1   & 0    
    }
    \end{equation}
    by elementary row and column operations. Since $\rank_{\mathbb{F}} I_G=4=N-(k-1)$, Lemma~\ref{lem:1d_local} asserts that $G$ is locally rigid in $\mathbb{F}^1$.
\end{eg}

\section{Identifiability of Segre Varieties and Generic Rigidity}\label{sec:iden}
In this section, we shall review useful tools from algebraic geometry for analyzing rigidity.
The following material is essentially a specialization of  
the discussion in \cite{cruickshank2023identifiability,sugiyama} to Segre varieties,
but several combinatorial interpretations based on special structures of tensor completions are new.
We shall also provide proofs of key propositions 
to explain the idea behind the statements.  


\paragraph{Identifiability.}
Let ${\cal V}$ be an affine variety in $\mathbb{C}^m$ and suppose that 
${\cal V}$ is a cone.
${\cal V}$ is called {\em $d$-identifiable} if any generic point 
$x$ in the $d$-secant $S_d({\cal V})$ of $\cal V$ can be written as 
$x=\sum_{k=1}^d x_k$ for unique $x_1,\dots, x_d\in {\cal V}$ up to scaling of each $x_i$ and permutation of indices.

The following proposition is the first main step in our analysis.  
It reduces the $d$-dimensional global rigidity problem to two questions: 
1-dimensional global rigidity and the $d$-identifiability of $\overline{\pi_G({\rm im}\ \sigma_{\bn})}$.
The remaining of this paper is solely devoted to how to solve those two questions.


\begin{prop}\label{prop:iden_global}
Let $\mathbb{F}\in \{\mathbb{R},\mathbb{C}\}$, $d$ be an integer with $d\geq 1$,
and
$G$ be a $k$-partite $k$-graph with $G\subseteq K_{\bn}^k$ and $N$ vertices.
Then $G$ is globally rigid in $\mathbb{F}^d$
if $G$ is globally rigid in $\mathbb{F}^1$ and $\overline{\pi_G({\rm im}\ \sigma_{\bn})}$ is $d$-identifiable over $\mathbb{C}$.
\end{prop}
\begin{proof}
Consider a generic point configuration $\bp \in (\mathbb{F}^d)^N$.
Note that $\bp$ is also a generic over $\mathbb{C}$ even if $\mathbb{F}=\mathbb{R}$.
Our goal is to show that an element in the fiber of $f_G^d$ at $f_G^d(\bp)$ is uniquely determined up to congruence by using the assumption of the statement.

Let $\bp_i$ be the $i$-th coordinate vector of $\bp$ for each $i$,
that is, the $n$-dimensional vector consisting of the $i$-th coordinates of the $n$ points of $\bp$.
Since $\bp$ is generic over $\mathbb{C}$, $f_G^d(\bp)$ is a generic point in the $d$-secant of $\overline{\pi_G({\rm im}\ \sigma_{\bn})}$ over $\mathbb{C}$.
Hence, the $d$-identifiability of $\overline{\pi_G({\rm im}\ \sigma_{\bn})}$ determines how $f_G^d(\bp)$ can be written as the sum of $d$ points in $\overline{\pi_G({\rm im}\ \sigma_{\bn})}$.
Those $d$ points are the images $f^1_G(\bp_1),\dots, f^1_G(\bp_d)$ of the $1$-dimensional rigidity map.
The global rigidity of $G$ in 1-dimensional space over $\mathbb{F}$ implies that 
the fiber of $f^1_G$ at $f^1_G(\bp_i)$ contains only $\bp_i$ up to congruence (over $\mathbb{F}$) for each $i$.
Thus, the fiber of $f_G^d$ at $f_G^d(\bp)$ is also uniquely determined up to congruence over $\mathbb{F}$.
\end{proof}

\paragraph{Tangential weak defectiveness.}
In view of Proposition~\ref{prop:iden_global},
checking the $d$-identifiability of $\overline{\pi_G({\rm im}\ \sigma_{\bn})}$ is a key to understand global rigidity.
To check $d$-identifiability, we shall use tangential weak defectiveness by Chiantini and Ottaviani~\cite{CO2012}.

Let us first define tangential weak defectiveness.
Let ${\cal V}\subseteq \mathbb{C}^m$ be the affine cone of a projective variety and $x_1,\dots, x_d$ be generic points in ${\cal V}$.
The {\em $d$-tangent contact locus} at $x_1,\dots, x_d$ is defined by
\[
C_{x_1,\dots, x_d}{\cal V}:=\overline{\{y\in {\cal V}: T_y{\cal V} \subset \langle T_{x_1} {\cal V}, \dots, T_{x_d} {\cal V}\rangle \}}.
\]
Since ${\cal V}$ is a cone, $\dim C_{x_1,\dots, x_d}{\cal V}\geq 1$.
${\cal V}$ is  {\em $d$-tangentially weakly defective} if 
an irreducible component of $C_{x_1,\dots, x_d}{\cal V}$ that contains $x_i$ has dimension at least two for some $i$.

Chiantini and Ottaviani~\cite{CO2012} showed that, if  ${\cal V}$ is not $d$-tangentially weakly defective, then ${\cal V}$ is $d$-identifiable.
By specializing it to our case, we have the following.
\begin{prop}\label{prop:twnd_iden}
If $\overline{\pi_G({\rm im}\ \sigma_n)}$ is not $d$-tangentially weakly defective,
then  $\overline{\pi_G({\rm im}\ \sigma_n)}$ is $d$-identifiable.
\end{prop}

\paragraph{Testing tangential weak non-defectiveness.}
We now adapt the linear algebraic test for
$d$-tangentially weakly non-defective due to 
Bocci, Chiantini, Ottaviani, and Vannieuwenhoven~\cite{BCC,COV2014,chiantini2017generic}. 
Interestingly, an adaptation of their idea to combinatorial projections of Segre varieties gives a new spectral graph theoretical condition.

Let $(G,\bp=(p_1,\dots, p_N))$ be a $d$-dimensional framework
with a  $k$-graph $G\subseteq K_{\bn}^k$ and $N$ vertices.
For an edge weight function $\omega:E(G)\rightarrow \mathbb{F}$,
the {\em first coordinated adjacency matrix} $A_{\omega}^1$ of $(G,\bp)$ is defined to be an $N\times N$ matrix whose
$(i,j)$-th entry is given by 
\[
A_{\omega}^1[i,j]=\begin{cases}
0 & (\text{if } i=j) \\
{\displaystyle \sum_{e\in E(G): \{i,j\}\subseteq e} \omega(e)
\left(\prod_{v\in e\setminus \{i,j\}} p_{v,1} \right)}
& (\text{otherwise})
\end{cases}
\qquad (i,j\in \{1,\dots, N\}),
\]
where $p_{v,1}$ denotes the first entry of $p_v\in \mathbb{F}^d$.

Since each row of $Jf_G^d(\bp)$ is indexed by an edge of $G$,
we may regard an element in the left kernel of $Jf_G^d(\bp)$ 
as an edge weight $\omega:E(G)\rightarrow \mathbb{F}$ of $G$.
We are particularly interested in the first coordinated adjacency matrices $A_{\omega}^1$ for an element $\omega$ in the left kernel of $Jf_G^d(\bp)$.

\begin{eg}\label{ex:A1}
Consider the same setting as that in Example~\ref{eg:1d_local}.
An edge weight $\omega=(\omega_e)_{e\in E(G)}$  of $G$ given by 
\begin{equation*}
    \kbordermatrix{
    & \omega_{135} & \omega_{136} & \omega_{146} & \omega_{236} & \omega_{245} \\
    & -\frac{1}{p_{1}p_{3}p_{5}} & \frac{2}{p_{1}p_{3}p_{6}} & -\frac{1}{p_{1}p_{4}p_{6}} & -\frac{1}{p_{2,1}p_{3}p_{6}} & \frac{1}{p_{2}p_{4}p_{5}}
    }
\end{equation*}
is in the left kernel of $Jf_G^d(\bp)$.
(Note that, in this example, $d=1$. So $p_{i,1}=p_i$.)
The corresponding $A_\omega^1$  is 
\begin{equation*}
A_{\omega}^1=
\kbordermatrix{
& 1 & 2 & 3 & 4 & 5 & 6 \\
1 & 0 & 0 & \omega_{135}p_{5} + \omega_{136}p_{6} & \omega_{146}p_{6} & \omega_{135}p_{3} & \omega_{136}p_{3} + \omega_{146}p_{4} \\
2 & 0 & 0 & \omega_{236}p_{6} & \omega_{245}p_{5} & \omega_{245}p_{4} & \omega_{236}p_{3} \\
3 & \omega_{135}p_{5}+\omega_{136}p_{6} & \omega_{236}p_{6} & 0 & 0 & \omega_{135}p_{1} & \omega_{136}p_{1}+\omega_{236}p_{2} \\
4 & \omega_{146}p_{6} & \omega_{245}p_{5} & 0 & 0 & \omega_{245}p_{2} & \omega_{146}p_{1} \\
5 & \omega_{135}p_{3} & \omega_{245}p_{4} & \omega_{135}p_{1} & \omega_{245}p_{2} & 0 & 0 \\
6 & \omega_{136}p_{3}+\omega_{146}p_{4} & \omega_{236}p_{3} & \omega_{136}p_1+\omega_{236}p_{2} & \omega_{146}p_{1} & 0 & 0
}.
\end{equation*}
\end{eg}

\begin{prop}\label{prop:lowerbound}
Let $(G,\bp)$ be a generic $d$-dimensional framework.
Then,
\begin{equation}\label{eq:kernel_lower}
\dim \left(\bigcap_{\omega\in \ker Jf_G^d(\bp)^{\top}} \ker A_{\omega}^1\right)\geq k,
\end{equation}
\end{prop}
\begin{proof}
For $j=1,\dots, k$, define the vector $\by_j^1\in \mathbb{F}^N$ by 
\[
\left(\by_j^1\right)_v=\begin{cases}
p_{v,1} & (v\in V_j) \\
0 & (\text{otherwise}).
\end{cases}
\]
The statement follows by showing that 
$\by_1^1,\dots, \by_k^1$ lie on the kernel of $A_{\omega}^1$
for any $\omega\in \ker Jf_G^d(\bp)^{\top}$.
This will be checked by directly computing 
$A_{\omega}^1 \by_j^1$.

Pick any $i\in\{1,\dots, k\}$ and $u\in V_i$.
For any $j\in \{1,\dots, k\}$, we have
\begin{align*}
(A_{\omega}^1 \by_j^1)_u&=
\sum_{v\in V_j} \left(\sum_{e\in E(G):\{u,v\}\in e} \omega(e)\left( \prod_{w\in e\setminus \{u,v\}} p_{w,1} \right) \right)p_{v,1} 
\\
&=\begin{cases}
0 & (\text{if } i=j) \\
    {\displaystyle \sum_{e\in E(G):u\in e} \omega(e)\left( \prod_{w\in e\setminus \{u\}} p_{w,1} \right)}  & (\text{if } i\neq j)
\end{cases}\\
&=0,
\end{align*}
where the first equation follows from the definitions of $A_{\omega}^1$ and $\by^1_j$,
the second equation follows from the fact that $G$ is $k$-partite,
and the third equation follows from the fact that $\omega$ is in the left kernel of $Jf_G^d(\bp)$.
Thus, $\by_j^1$ is indeed in the kernel of $A_{\omega}$.

Since $\by_j^1,\dots, \by_k^1$ are linearly independent,
the statement follows.
\end{proof}

\begin{lemma}\label{lem:COtest}
Let $(G,\bp)$ be a generic $d$-dimensional framework with $G\subseteq K_{\bn}^k$.
If
\begin{equation}\label{eq:kernel}
\dim \left(\bigcap_{\omega\in \ker Jf_G^d(\bp)^{\top}} \ker A_{\omega}^1\right)=k,
\end{equation}
then $\overline{\pi_G({\rm im} \ \sigma_{\bn})}$ is $d$-tangentially weakly non-defective over $\mathbb{C}$.
\end{lemma}
\begin{proof}
This is an adaptation of the result of Bocci, Chiantini, Ottaviani, and Vannieuwenhoven~\cite{BCC,COV2014,chiantini2017generic}.

For simplicity, denote ${\cal S}_G=\overline{\pi_G({\rm im}\ \sigma_{\bn}})$.
Let $\bp_1,\dots, \bp_d$ be the coordinate vectors of $\bp$
and let $x_i=f^1_G(\bp_i)$ for $i=1,\dots, d$.
Since $\bp$ is generic, each $x_i$ is a generic point of ${\cal S}_G$.
So, the tangent space $T_{x_i}{\cal S}_G$ at $x_i$ is the image of 
$Jf_G^1(\bp_i)$.
Thus, the $d$-tangent contact locus of ${\cal S}_G$ at $x_1,\dots, x_d$ is written as 
\[
C_{x_1,\dots, x_d}{\cal S}_G
=\overline{\{f_G^1(\bq): \bq\in \mathbb{C}^N,  
{\rm im}\ Jf_G^1(\bq)\subset\langle {\rm im}\ Jf_G^1(\bp_i): i=1,\dots, d \rangle\}}
\]
Moreover, since $Jf_G^d(\bp)$ is obtained by aligning 
$Jf_G^1(\bp^i)$ for $i=1,\dots, d$, 
this can be simplified to
\begin{equation}\label{eq:locus}
C_{x_1,\dots, x_d}{\cal S}_G
=\overline{\{f_G^1(\bq): \bq\in \mathbb{C}^N,  
{\rm im}\ Jf_G^1(\bq)\subset {\rm im}\ Jf_G^d(\bp)\}}.
\end{equation}
Suppose (\ref{eq:kernel}) holds.
Then there are $\omega_1,\dots, \omega_t\in \ker Jf_G^d(\bp)^{\top}$
such that 
\begin{equation}\label{eq:kernel2}
\dim \left(\bigcap_{i=1}^t \ker A_{\omega_i}^1\right)=k.
\end{equation}
By (\ref{eq:locus}) and $\omega_j\in \ker Jf_G^d(\bp)^{\top}$,
\begin{equation}\label{eq:locus2}
C_{x_1,\dots, x_d}{\cal S}_G\subseteq 
\overline{\{f_G^1(\bq): \bq\in \mathbb{C}^N, \omega_j\in \ker Jf_G^1(\bq)^{\top}, j=1,\dots, t\}}.
\end{equation}
Now, instead of computing the dimension of the irreducible component of $C_{x_1,\dots, x_d}{\cal S}_G$ at $x_1$,
we compute that in the right set in (\ref{eq:locus2}).
By directly computing the entries of $Jf_G^1$, 
one can verify that 
$\bq=(q_1,\dots, q_N)\in \mathbb{C}^N$ satisfies $\omega_j\in Jf_G^1(\bq)^{\top}$ 
if and only if 
\begin{equation}\label{eq:equiliblium}
\sum_{e\in E(G): u\in e} \omega_j(e)\left(\prod_{v\in e\setminus\{u\}}q_v \right)=0\qquad (u\in V(G)).
\end{equation}
This is a polynomial system in $\bq$, so 
it is written as $h_{\omega_j}(\bq)=0$ for some 
polynomial map $h_{\omega_j}:\mathbb{C}^N\rightarrow \mathbb{C}^N$ (determined by $\omega_j$).
It turns out from (\ref{eq:equiliblium}) that the Jacobian $Jh_{\omega_j}(\bp_1)$ of $h_{\omega_j}$ at $\bp_1$ is exactly $A_{\omega_j}^1$.

Now, by (\ref{eq:kernel2}),
the kernel of the Jacobian of the polynomial map $h:=(h_{\omega_1},\dots, h_{\omega_t})$ is $k$-dimensional.
Since $\bp_1$ is non-singular in the solution space of $h$,
the irreducible component of $\bp_1$ in the vanishing set of $h$ is $k$-dimensional.
Since the kernel of $f_G^1$ is at least $(k-1)$-dimensional
(as the value of $f_G^1$ is invariant by the action of stabilizers), 
this in turn implies that the right set in (\ref{eq:locus2}) is one-dimensional in the irreducible component of $x_i$,
and the same property holds in $C_{x_1,\dots, x_d}{\cal S}_G$ as well. (Note that ${\cal S}_G$ is a cone, so every irreducible component has dimension at least one.)
\end{proof}

Combining the results so far, we have the following.
\begin{theorem}\label{thm:COtest}
Let $\mathbb{F}\in \{\mathbb{C}, \mathbb{R}\}$ and $G$ be a $k$-uniform $k$-graph with $G\subseteq K_{\bn}^k$.
Suppose 
\begin{itemize}
\item[(i)] $G$ is globally rigid in $\mathbb{F}^1$, and
\item[(ii)] there is a generic $d$-dimensional point configuration
$\bp:(\mathbb{F}^d)^N$ 
satisfying (\ref{eq:kernel}).
\end{itemize}
Then $G$ is globally rigid in $\mathbb{F}^d$.
\end{theorem}
 \begin{proof}
 Note that the condition (ii) is irrelevant to whether $\mathbb{F}$ is $\mathbb{R}$ or $\mathbb{C}$.
 Hence, we may assume that the condition (ii) holds over $\mathbb{C}$.
 Then, by Lemma~\ref{lem:COtest} and Proposition~\ref{prop:twnd_iden},
 $\overline{\pi_G({\rm im}\ {\sigma}_d)}$ is $d$-identifiable.
 By the condition (i) and Proposition~\ref{prop:iden_global},
the $d$-dimensional global rigidity of $G$ over $\mathbb{F}$ follows.
 \end{proof}

\paragraph{Masaratti-Mella's sufficient condition.}
Masaratti and Mella~\cite{MMident} have recently shown a much simpler sufficient condition for $d$-identifiability.
Interesting, 
Masaratti-Mella's sufficient condition leads to a purely graph-theoretical condition
for our particular varieties. 

Let $G$ be a $k$-uniform $k$-graph with $G\subseteq K_{\bn}^k$.
For an edge weight function $\omega:E(G)\rightarrow \mathbb{F}$,
the {\em adjacency matrix} $A_{\omega}$ of $G$ is defined to be an $N\times N$ matrix whose
$(i,j)$-th entry is given by 
\[
A_{\omega}[i,j]=\begin{cases}
0 & (\text{if } i=j) \\
{\displaystyle 
\sum_{e\in E(G): \{i,j\}\subseteq e} \omega(e)}
& (\text{otherwise}) 
\end{cases}
\qquad (i,j\in \{1,\dots, N\}).
\]

\begin{theorem}\label{thm:MM_test}
Let $\mathbb{F}\in \{\mathbb{C},\mathbb{R}\}$
and $G$ be a $k$-uniform $k$-graph with $G\subseteq K_{\bn}^k$.
Suppose
\begin{itemize}
\item[(i)] $G$ is globally rigid in $\mathbb{F}^1$,
\item[(ii)] $G$ is locally rigid in $\mathbb{C}^{d+1}$, and 
\item[(iii)] $G$ satisfies
\begin{equation}\label{eq:simplified_kernel}
\dim \left(\bigcap_{\omega\in \ker I_G} \ker A_{\omega}\right)=k.
\end{equation}
\end{itemize}
Then $G$ is globally rigid in $\mathbb{F}^d$.
\end{theorem}
\begin{proof}
By Proposition~\ref{prop:iden_global} together with (i),
it suffices to show that 
$\overline{\pi_G({\rm im} \ \sigma_{\bn})}$ is $d$-identifiable.
We apply the Masaratti and Mella theorem~\cite[Theorem 1.5]{MMident} to $\overline{\pi_G({\rm im} \ \sigma_{\bn})}$.
The theorem implies that 
$\overline{\pi_G({\rm im} \ \sigma_{\bn})}$ is $d$-identifiable 
if 
\begin{description}
    \item[(MM1)] the $(d+1)$-secant of $\overline{\pi_G({\rm im} \ \sigma_{\bn})}$ has dimension $(d+1)\left(N-(k-1)\right)$, and
    \item[(MM2)] $\overline{\pi_G({\rm im}\ \sigma_{\bn})}$ is $1$-tangentially weakly non-defective.
\end{description}
By Proposition~\ref{prop:infinitesimal}, the condition (MM1) is equivalent to the generic $(d+1)$-dimensional local rigidity of $G$, which is the condition (ii) in the statement.

We show that the condition (MM2) is implied by the condition (iii) in the statement.
By Lemma~\ref{lem:COtest}, the condition (MM2) holds if 
\begin{equation}\label{eq:kernel_1}
\dim \left(\bigcap_{\omega\in \ker Jf_G^1(\bp)^{\top}} \ker A_{\omega}^1\right)=k
\end{equation}
for a generic $1$-dimensional framework $(G,\bp)$ of $G$.
(Note that here  $\bp=(p_1,\dots, p_N)$ is a $1$-dimensional point configuration. So each $p_i$ is a scalar.)
We have seen in the proof of Lemma~\ref{lem:1d_local}
that $Jf_G^1(\bp)^{\top}$ can be transformed to $I_G$ by fundamental row and column operations.
This gives a linear bijection  between 
$\ker Jf_G^1(\bp)^{\top}$ and $\ker I_G$.
More explicitly, this bijection sends $\omega\in \ker Jf_G^1(\bp)^{\top}$ to $\omega'\in I_G$ written by
\begin{equation}\label{eq:omega_bijection}
\omega'(e)=\omega(e)\left(\prod_{v\in e} p_v\right)\qquad (e\in E(G)).
\end{equation}
Using the definitions of $A_{\omega}^1$ and $A_{\omega}$, 
one can directly check that 
\[
A_{\omega'}=\begin{pmatrix}
p_1 & 0 &  \cdots & 0 \\
0 & p_2 &  \cdots & 0 \\
\vdots & \vdots & \ddots & \vdots \\
0 & 0 &  \cdots & p_N
\end{pmatrix} A_{\omega}^1
\begin{pmatrix}
p_1 & 0 &  \cdots & 0 \\
0 & p_2 &  \cdots & 0 \\
\vdots & \vdots & \ddots & \vdots \\
0 & 0 &  \cdots & p_N
\end{pmatrix}
\]
for any $\omega\in \ker Jf_G^1(\bp)^{\top}$ and the corresponding $\omega'\in I_G$ in (\ref{eq:omega_bijection}).
Thus, (\ref{eq:kernel_1}) is equivalent to (\ref{eq:simplified_kernel}),
and (MM2) follows.
\end{proof}

\begin{eg}\label{ex:A2}
Consider the same setting as Example~\ref{eg:1d_local}.
An edge weight $\omega=(\omega_e)_{e\in E(G))}$  of $G$ defined by 
\begin{equation*}
        \kbordermatrix{
        & (1,3,5) & (1,3,6) & (1,4,6) & (2,3,6) & (2,4,5) \\
        & -1 & 2 & -1 & -1 & 1
        },
    \end{equation*}
    is in the (right) kernel of $I_G$ given in (\ref{eq:IGexample}). Then,
\begin{equation*}
A_{\omega}=
\kbordermatrix{
& 1 & 2 & 3 & 4 & 5 & 6 \\
1 & 0 & 0 & 1 & -1 & -1 & 1 \\
2 & 0 & 0 & -1 & 1 & 1 & -1 \\
3 & 1 & -1 & 0 & 0 & -1 & 1 \\
4 & -1 & 1 & 0 & 0 & 1 & -1 \\
5 & -1 & 1 & -1 & 1 & 0 & 0 \\
6 & 1 & -1 & 1 & -1 & 0 & 0
},
\end{equation*}
whose rank is $3=N-k$.
Hence, $G$ satisfies the condition (iii) of Theorem~\ref{thm:MM_test}.
\end{eg}

\begin{eg}\label{ex:K}
Let us show that $K_{(2,2,1,\dots, 1)}^k$ satisfies the condition (iii) of Theorem~\ref{thm:MM_test}.
Denote the vertices by $V_1=\{v_{1,1},v_{1,2}\}, V_2=\{v_{2,1},v_{2,2}\}, V_i=\{v_i\}$ for $i=3,\dots, k$
and the edges by
$e_{1,1}=(v_{1,1},v_{2,1},v_3\dots, v_k)$,
$e_{1,2}=(v_{1,1},v_{2,2},v_3\dots, v_k)$,
$e_{2,1}=(v_{1,2},v_{2,1},v_3\dots, v_k)$,
$e_{1,1}=(v_{1,2},v_{2,2},v_3\dots, v_k)$.
Set an edge weight $\omega$ by
$\omega(e_{1,1})=1, \omega(e_{1,2})=-1, \omega(e_{2,1})=-1, \omega(e_{2,2})=1$.
Then $\omega$ is in the kernel of $I_{K_{(2,2,1,\dots, 1)}^k}$
and 
\begin{equation*}
A_{\omega}=
\kbordermatrix{
& v_{1,1} & v_{1,2} & v_{2,1} & v_{2,2} & v_3 & \dots & v_{k} \\
v_{1,1} & 0 & 0 & 1 & -1 & 0 & \dots & 0 \\
v_{1,2} & 0 & 0 & -1 & 1 & \vdots &  & \vdots \\
v_{2,1} & 1 & -1 & 0 & 0 & \vdots &  & \vdots \\
v_{2,2} & -1 & 1 & 0 & 0 & \vdots &  &  \vdots  \\
v_3 & 0 & \dots & \dots & \dots & 0 &  & \vdots \\
\vdots & \vdots &  &  &  &  & \ddots & \vdots \\
v_k & 0 & \dots & \dots & \dots  & \dots &  & 0
}.
\end{equation*}
We have $\rank A_{\omega}=2=N-k$, implying that 
$K_{(2,2,1,\dots, 1)}^k$ satisfies the condition (iii) of Theorem~\ref{thm:MM_test}.
\end{eg}
\section{Checking One-dimensional Generic Global Rigidity}\label{sec:1d}
In view of Theorem~\ref{thm:COtest} or Theorem~\ref{thm:MM_test}, 
solving the 1-dimensional global rigidity problem is a crucial step for understanding general dimensional global rigidity.
In this section, we shall focus on the 1-dimensional case.

We first solve the case when $\mathbb{F}=\mathbb{R}$ since this case is simpler but still captures the main idea.

\begin{theorem}\label{thm:1d_global_real}
Let $G$ be a $k$-partite $k$-graph with $N$ vertices.
Then $G$ is globally rigid in $\mathbb{R}^1$ if and only if
$\rank_{\mathbb{R}} I_G=\rank_{GF(2)} I_G=N-(k-1)$.
\end{theorem}
\begin{proof}
Let us consider a 1-dimensional framework $(G,\bp)$
with a generic 1-dimensional point configuration $\bp=(p_1,\dots, p_N)\in \mathbb{R}^N$.
Our goal is to show that, for any $\bq=(q_1,\dots, q_N)\in \mathbb{F}^N$, $\bq$ is congruent to $\bp$ 
whenever $f_G^1(\bp)=f_G^1(\bq)$.
Observe that, by specializing the definition of stabilizers to the 1-dimensional case, $\bq$ is congruent to $\bp$ if 
there are $a_1,\dots, a_k\in \mathbb{R}$ satisfying $\prod_{i=1}^k a_i=1$ such that 
$q_v=a_i p_v\ (v\in V_i, i=1,\dots, k)$.

Now, the condition $f_G^1(\bp)=f_G^1(\bq)$ is written as
\begin{equation}\label{eq:1d_real0}
\prod_{v\in e} p_{v} = \prod_{v\in e} q_{v}
\qquad (e\in E(G)).
\end{equation}
We can write
$p_v$ and $q_v$ as $p_v=s_v^p r_v^p$
and $q_v=s_v^q r_v^q$ using $s_v^p, s_v^q\in \{-1,1\}$ and $r_v^p, r_v^q> 0$
for each $v\in V(G)$.
Then (\ref{eq:1d_real0}) is equivalent to
\begin{equation}\label{eq:1d_real1}
\sum_{v\in e}^k(\log r_{v}^p-\log r_{v}^q)=0
\quad \text{ and } \quad
\prod_{v\in e} s_{v}^p = \prod_{v\in e} s_{v}^q
\qquad (e\in E(G))
\end{equation}
or equivalently
\begin{equation}\label{eq:1d_real2}
I_G^{\top} \bx = 0 \quad \text{ and } \quad I_G^{\top} \by \equiv 0 \mod 2
\end{equation}
by setting
$x_v = \log r_v^p-\log r_v^q$, 
and 
$y_v = 0$ if $s_v^p=s_v^q$ and otherwise $y_v = 1$.
Hence, 
$\rank_{\mathbb{R}}I_G=\rank_{GF(2)} I_G=n-(k-1)$ holds if and only if
any solutions $\bx, \by$ of (\ref{eq:1d_real2}) are in the trivial kernel of $I_G$, that is, $\bx, \by$ are written as 
\[
x_v = \alpha_i \text{ and }
y_v = \beta_i \text{ for each $v\in V_i$  }
\]
for some $\alpha_1,\dots, \alpha_k\in \mathbb{R}, \beta_1,\dots, \beta_k\in \{0,1\}$ with $\sum_{i=1}^k \alpha_i=0$ and $\sum_{i=1}^k \beta_i\equiv 0 \mod 2$
(see Section~\ref{subsec:1d_local} for the definition of the trivial kernel).
This condition is equivalent that $\bq$ is written as $q_v=a_i p_v\ (v\in V_i, i=1,\dots, k)$
for some $a_1,\dots, a_k\in \mathbb{R}$ satisfying $\prod_{i=1}^k a_i=1$
by setting $a_i=(-1)^{\beta_i} \exp(\alpha_i)$ for each $i$.
\end{proof}

We now extend this argument to the complex case.
The argument seems to be extendable for a complete characterization even in the complex case, but here we shall focus on giving a simpler sufficient condition for subsequent applications.

\begin{theorem}\label{thm:1d_global_complex}
Let $G$ be a $k$-partite $k$-graph with $N$ vertices.
Then $G$ is globally rigid in $\mathbb{C}^1$ if 
$\rank_{\mathbb{R}} I_G=\rank_{GF(q)} I_G=N-(k-1)$
for any prime $q$ with $q\leq k^{N/2}$.
\end{theorem}
\begin{proof}
As in the proof of Theorem~\ref{thm:1d_global_real}, let us consider a 1-dimensional framework $(G,\bp)$
with a generic 1-dimensional point configuration $\bp=(p_1,\dots, p_N)\in \mathbb{C}^N$,
and pick any $\bq=(q_1,\dots,q_N)$ satisfying  $f_G^1(\bp)=f_G^1(\bq)$.
We denote $p_v=r_v^p \exp(\mathrm{i}\theta_v^p)$ and 
$q_v=r_v^q \exp(\mathrm{i}\theta_v^q)$ by preparing real numbers $r_v^p, r_v^q> 0$
and $\theta_v^p, \theta_v^q\in [0,2\pi)$ for each $v\in V(G)$.
Then, the condition $f_G^1(\bp)=f_G^1(\bq)$ is equivalent to
\begin{align}\label{eq:1d_complex1}
\sum_{v\in e}(\log r_{v}^p-\log r_{v}^q)&=0 \qquad (e\in E(G)), \\
\sum_{v\in e} (\theta_v^p-\theta_v^q)&\in 2\pi\mathbb{Z}
\qquad (e\in E(G)). \label{eq:1d_complex2}
\end{align}
As we have seen in the proof of Theorem~\ref{thm:1d_global_real},
(\ref{eq:1d_complex1}) determines $r_v^q$ up to the difference of a trivial kernel vector if and only if $\rank_{\mathbb{R}} I_G=N-(k-1)$. 
So we can reuse this observation to see that 
that $\log r_v^p-\log r_v^q=\alpha_i\ (v\in V_i)$
for some $\alpha_1,\dots, \alpha_k>0$.
The new observation of this proof is the analysis of (\ref{eq:1d_complex2}), which is given in the subsequent discussion.

Let $V_1,\dots, V_k$ be the $k$-partition of $V(G)$,
and pick any representative vertex $v_i$ for each $V_i$.
We define the augmented incidence matrix $\tilde{I}_G$ to be that obtained from $I_G$ by appending the characteristic vectors $\chi_{v_i}$ of $v_i$ for all $i=2,\dots, k$ as new columns.
Since no trivial kernel vector is orthogonal to those characteristic vectors, 
$\rank_{\mathbb{R}} \tilde{I}_G=\rank_{GF(q)} \tilde{I}_G=N$ follows 
from $\rank_{\mathbb{R}} I_G=\rank_{GF(q)} I_G=N-(k-1)$.
Thus, $\tilde{I}_G$ has a non-singular square submatrix  $I_0$ of size $N$  over $\mathbb{R}$.

Now, we prepare a vector $\bz$ so that 
\begin{equation}\label{eq:1d_complex7}
z_v=\frac{\det(I_0)(\theta_v^p-\theta_v^q)}{2\pi} \qquad (v\in V(G)).
\end{equation}
Then (\ref{eq:1d_complex2}) is equivalent to
\begin{equation}\label{eq:1d_complex3}
I_G^{\top}\bz \in  (|\det(I_0)|\mathbb{Z})^m.
\end{equation}
By Cramer's rule, $\bz$ is an integer vector, and hence
(\ref{eq:1d_complex3}) gives
\begin{equation}\label{eq:1d_complex4}
I_G^{\top}\bz \equiv 0 \mod |\det(I_0)|.
\end{equation}
Consider the vectors $\bx_2,\dots, \bx_k$ defined in (\ref{eq:1d_trivial}), which form a basis of the trivial kernel of $I_G$.
By setting 
\begin{equation}\label{eq:1d_complex5}
\tilde{\bz}=\bz+\sum_{i=2}^k z_{v_i}\bx_i,
\end{equation}
we have 
\begin{equation}\label{eq:1d_complex6}
\tilde{I}_G^{\top}\tilde{\bz} \equiv 0 \mod |\det(I_0)|.
\end{equation}
from (\ref{eq:1d_complex4}).
This in turn implies that $\tilde{I}_G^{\top}\tilde{\bz} \equiv 0 \mod q$
for any prime divisor of $|\det(I_0)|$.
Hadamard's inequality on matrix determinants gives
$|\det(I_0)|\leq k^{N/2}$.
Therefore, the assumption in the statement gives $\rank_{GF(q)} \tilde{I}_G=n$
for any prime divisor $q$ of $|\det(I_0)|$.
We thus obtain $\tilde{\bz}\equiv 0 \mod q$ for any prime divisor $q$ of $|\det(I_0)|$. 
By the Chinese remainder theorem, this gives  
$\tilde{\bz}\equiv 0 \mod |\det(I_0)|$.
Combining this with (\ref{eq:1d_complex7}) and (\ref{eq:1d_complex5}), we 
conclude that there are some scalars $\beta_1,\dots, \beta_k$ such that
\[
\theta_v^p-\theta_v^q+\beta_i\in 2\pi\mathbb{Z} \quad (v\in V_i).
\]

Recall that $\log r_v^p-\log r_v^q=\alpha_i\ (v\in V_i)$
for some $\alpha_1,\dots, \alpha_k>0$.
Hence,
$\bq$ is written as $q_v=a_i p_v\ (v\in V_i, i=1,\dots, k)$
for some $a_1,\dots, a_k\in \mathbb{C}$ satisfying $\prod_{i=1}^k a_i=1$
by setting $a_i=\exp(\alpha_i)\exp(i\beta_i)$ for each $i$,
and we conclude that $\bq$ is congruent to $\bp$.
Thus, $(G,\bp)$ is globally rigid.
\end{proof}

\section{Polymatroids on Random Subgraphs}\label{sec:random}
Throughout this section, we focus on the case when 
$\bn=(n,\dots, n)$ and  denote $K_{\bn}^k$  by $K_{n}^k$.

 We first introduce a graph operation.
 Given a $k$-uniform $k$-graph $G$, a {\em degree-$d$ extension}  creates a new $k$-graph from $G$ by adding a new vertex $v$ and $d$ distinct new hyperedges,
 each of which consists of $v$ and $k-1$ vertices from $G$.
 A $k$-partite graph is said to be a {\em $k$-partite $d$-tree}\footnote{The term is motivated from {\em partial $d$-trees} for undirected graphs.} if
 it is isomorphic to a $k$-graph which can be built from $K_d^k$ by a sequence of degree-$d$ extensions.
See Figure~\ref{fig:3tree} for an example.

\begin{figure}[t]
\centering
\includegraphics[scale=0.5]{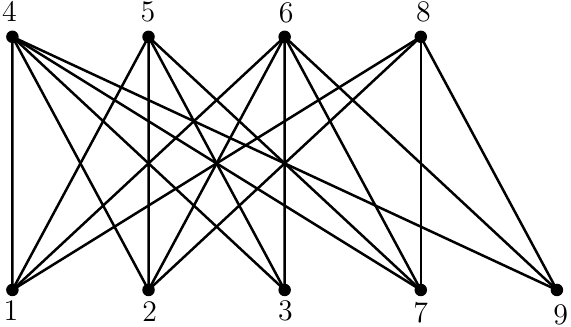}
\caption{An example of a $2$-partite $3$-tree ($k=2$ and $d=3$).
It can be build from $K_3^2$ on $\{1,2,3\}\cup\{4,5,6\}$ by adding $7, 8, 9$ by degree-$3$ extensions.}
\label{fig:3tree}
\end{figure}

Define a random variable
\[
M_d=\min \{m: \text{the minimum degree of $G(n,m)$ is at least $d$}\}
\]
 Our goal in this section is to prove the following theorem.
 \begin{theorem}\label{thm:random}
 Let $a,d,k$ be positive integers.
 Suppose $P_n$ is a polymatroid on the edge set of $K_n^k$ with rank at most $an$ for each $n\in \mathbb{N}$.
 Then, a.a.s.~the graph of the closure of $G(n,M_d)$ in $P_n$ contains a spanning $k$-partite $d$-tree.
 \end{theorem}

 The proof of Theorem~\ref{thm:random} is an adaptation of that of 
 Lew et al.~\cite{lew2023sharp}, where they proved the corresponding result for the Euclidean graph rigidity matroids on complete graphs.
 Since Theorem~\ref{thm:random} deals with hypergraphs and moreover non-complete, we need a careful check if the argument of Lew et al.~can be adapted.

 As in the proof in \cite{lew2023sharp}, the proof consists of the following two main observations.

\begin{lemma}\label{lem:clique}
 Let $P_n$ be a polymatroid on the edge set of $K_n^k$ with rank at most $an$, and $\gamma$ be a constant with $0<\gamma<1$.
 Then there is a constant $c$ (which depends on $a, \gamma, k, d$ but does not on $n$) such that, 
 with probability at least $1-\left(\frac{1}{2k}\right)^{kn}$,
 the graph of the closure of $G(n, cn)$ in $P_n$ contains a $k$-partite $d$-tree whose vertex set has size at least $(1-\gamma)n$ on each side.
\end{lemma}

Let $G$ be a $k$-partite graph.
For a vertex set $X\subseteq V(G)$, a vertex $v\in X\cap V(G)$ is said to be {\em $d$-extendable} 
with respect to $X$ if $G$ has at least $d$ hyperedges consisting of 
$\{v\}$ and vertices in $V(G)\setminus X$.
A vertex set $X$ is said to be {\em $d$-extendable} in $G$ if 
$X$ contains a $d$-extendable vertex in $G$.

\begin{lemma}\label{lem:extension}
Asymptotically almost surely, $G(n,M_d)$ has the property that 
every vertex set $B$ with 
$|B\cap V_i| \leq \frac{n}{k}$ for all $i\in \{1,\dots, k\}$  is $d$-extendable in $G(n,M_d)$.
\end{lemma}

The proofs of Lemmas~\ref{lem:clique} and~\ref{lem:extension}
are provided in the next two subsections, respectively.

\begin{proof}[Proof of Theorem~\ref{thm:random}.]
By Lemma~\ref{lem:clique},
for an appropriate choice of a constant $c$,
with high probability 
the graph of the closure of $G(n,cn)$ in $P_n$ contains a $k$-partite $d$-tree $T$ whose  vertex set has size at least $(1-\gamma)n$ on each side.
By Proposition~\ref{prop:min_degree}, $M_d\gg cn$ if $n$ is sufficiently large.
Hence, for sufficiently large $n$, we may assume $G(n,cn)\subseteq G(n,M_d)$
and $T$ is in the closure of $G(n,M_d)$.
By Lemma~\ref{lem:extension}, a.a.s.~$T$ can be extended to a spanning $k$-partite $d$-tree $T^*$ by a sequence degree-$d$ extensions using only hyperedges of $G(n,M_d)$.
The resulting $T^*$ is contained in the closure of $G(n,M_d)$.
\end{proof}

\subsection{Proof of Lemma~\ref{lem:clique}}

Recall that ${\rm cl}_{P_n}$ denotes the closure operator of $P_n$.
We shall abuse the notation slightly and 
denote ${\rm cl}_{P_n}(E(G))$  by ${\rm cl}_{P_n}(G)$ for $G\subseteq K_n^k$.

\begin{lemma}\label{lem:closure_size}
Let $P_n$ be a polymatroid on the edge set of $K_n^k$ with rank at most $an$, and $\delta$ be any constant with 
$0<\delta<1$.
Then there is a constant $c$ (which depends on $a, \delta, k, d$ but does not on $n$) such that the closure of $G(n,cn)$ in $P_n$ has size at least $(1-\delta)n^k$ with probability at least $1-\left(\frac{1}{2k}\right)^n$.
\end{lemma}
\begin{proof}
The proof is identical to \cite[Lemma 3.1]{lew2023sharp} (with a quantitative evaluation of the success probability for an application in Theorem~\ref{thm:1d_global_complex}).

Consider the stochastic process for constructing $G(n,cn)$, that is, 
starting from the $k$-graph $G(n,0)$ with $n$ vertices and no hyperedge,
construct $G(n,i)$ from $G(n,i-1)$ by adding one hyperedge $e_i$ from $E(K_n^k)\setminus E(G(n,i-1))$ uniformly at random.
Then, at each step, the size of the closure in $P_n$ changes if and only if
$e_i$ does not belong to ${\rm cl}_{P_n}(G(n,i-1))$.
Hence, the probability that this event happens is 
$\frac{n^k-|{\rm cl}_{P_n}(G(n,i-1))|}{n^k-(i-1)}$.
As long as $|{\rm cl}_{P_n}(G(n,i-1))|< (1-\delta)n^k$, this probability is further bounded as 
\[
\frac{n^k-|{\rm cl}_{P_n}(G(n,i-1))|}{n^k-(i-1)}\geq \delta.
\]

Since the rank of $P_n$ is bounded by $an$, the rank can change at most $an$ times during the whole process. Hence, 
\begin{align*}
&\mathbb{P}[|{\rm cl}_{P_n}(G(n,cn))|< (1-\delta)n^k] \\
&= \mathbb{P}[\text{the rank changes at most $an$ times and }  
|{\rm cl}_{P_n}(G(n,i))|<(1-\delta)n^k \text{ for } i=0, 1, \dots, cn] \\
&\leq 
\mathbb{P}[|\{j: j\leq cn, e_j\notin {\rm cl}_{P_n}(G(n,j-1)\}|\leq an \text{ and }  
|{\rm cl}_{P_n}(G(n,i))|<(1-\delta)n^k \text{ for } i=0, 1, \dots, cn-1] \\
&\leq  \mathbb{P}[{\rm Bin}(cn,\delta)\leq an]. 
\end{align*}

In order to evaluate the last term, 
we use the Chernoff bound for Poisson trials,
which states that,
for any $\varepsilon$ with $0< \varepsilon<1$,
\[
\mathbb{P}[{\rm Bin}(cn,\delta)\leq (1-\varepsilon)\delta cn] 
\leq e^{-\delta cn \varepsilon^2/2}.
\]
Hence, if $c$ is set to be sufficiently large so that 
\[
\delta c \varepsilon^2 /2 \geq k\log(2k) \text{ and }  (1-\varepsilon)\delta c \geq a,
\]
we have
\begin{align*}
\mathbb{P}[{\rm Bin}(cn,\delta)\leq an]
\leq \mathbb{P}[{\rm Bin}(cn,\delta)\leq (1-\varepsilon)\delta cn] 
\leq e^{-\delta cn \varepsilon^2/2} 
\leq \left( \frac{1}{2k}\right)^{kn}.
\end{align*}
\end{proof}

Lemma~\ref{lem:closure_size} only says that 
the size of ${\rm cl}_{P_n}(G(n,cn))$ is large,
whereas  our goal, Lemma~\ref{lem:clique}, states that the graph of ${\rm cl}_{P_n}(G(n,cn))$ actually contains a large $d$-tree.
We first show the existence of a small clique with a nice expansion property. 
For this, we now consider ${\rm cl}_{P_n}(G(n,cn))$ as a pure $(k-1)$-dimensional abstract simplicial complex and investigate properties of lower dimensional faces.
To avoid confusion, let $\Delta$ be the corresponding abstract simplicial complex, that is, the collection of all sets contained in some hyperedge in ${\rm cl}_{P_n}(G(n,cn))$.
Also, let $V_1,\dots, V_k$ be the partition of the vertex set of $K_{n}^k$.
For $I\subseteq \{1,\dots, k\}$, 
let 
\[
K_I:=\left\{\sigma: \sigma\subset \bigcup_{i\in I} V_i, |\sigma\cap V_i|=1 \text{ for } i\in I\right\}.
\]
Clearly, $|K_I|=n^{|I|}$.

Let $t$ be a constant with $0<t<1$ (which will be determined later).
For each $I\subseteq \{1,\dots, k\}$,
we define $\Delta_I^{(t)}\subseteq \Delta$ recursively as follows:
\begin{itemize}
\item for $I=\{1,\dots, k\}$, $\Delta_I^{(t)}:={\rm cl}_{P_n}(G(n,cn))$,
\item for $I\subsetneq \{1,\dots, k\}$ and $j\in \{1,\dots, k\}\setminus I$, 
\[
\Delta_{I,j}^{(t)}:=\left\{\sigma\in K_{I}: |\{v\in V_j: \sigma\cup\{v\}\in \Delta_{I\cup\{j\}}^{(t)}\}|\geq (1-t)n\right\}, and
\]
\item for $I\subsetneq \{1,\dots, k\}$,
\[
\Delta_I^{(t)}:=\bigcap_{j\in \{1,\dots, k\}\setminus I} \Delta_{I,j}^{(t)}.
\]
\end{itemize}
\begin{lemma}\label{lem:good_faces}
Let $\Delta$ be defined as above.
Let $\delta_k$ be a constant with $0<\delta_k<1$ and define $\delta_i$ by
\[
\delta_i=\frac{k\delta_{i+1}}{t} 
\]
for $i=k-1, k-2, \dots, 1, 0$.
Suppose that $\Delta_{\{1,\dots,k\}}^{(t)}\geq (1-\delta_k)n^k$ and $\delta_0<1$. Then for any $I\subseteq \{1,2,\dots, k\}$
\[
|\Delta_I^{(t)}|\geq (1-\delta_{|I|}) n^{|I|}.
\]
\end{lemma}
\begin{proof}
The proof is done by induction on $|I|$.
The base case when $I=\{1,\dots, k\}$ follows from the assumption that 
$\Delta_{\{1,\dots,k\}}^{(t)}\geq (1-\delta_k)n^k$.

Consider the case of $I\subsetneq \{1,\dots, k\}$ assuming the statement for any $I'$ with $|I'|>|I|$.
For $\sigma\in K_I$ and $j\in \{1,\dots, k\}\setminus I$, let $N_{j}(\sigma):=\{v\in V_j: \sigma+v\in \Delta^{(t)}_{I\cup\{j\}}\}$.
Then by the definition of $\Delta_{I,j}^{(t)}$,
$\sigma\in K_I$ is not in $\Delta_{I,j}^{(t)}$ if and only if 
$|N_{j}(\sigma)|<(1-t)n$, or equivalently $tn< n-|N_{j}(\sigma)|$.
Taking the sum of this inequality over all $\sigma\in K_I\setminus \Delta_{I,j}^{(t)}$, we obtain
\begin{align*}
(n^{|I|}-|\Delta_{I,j}^{(t)}|)tn &\leq \sum_{\sigma\in K_I\setminus \Delta_{I,j}^{(t)}} (n-|N_j(\sigma)|) \\
&\leq \sum_{\sigma\in K_I} (n-|N_j(\sigma)|) \qquad \text{(by $n\geq |N_j(\sigma)|$)}\\
&=n^{|I|+1}-|\Delta_{I+j}^{(t)}| \qquad \text{(by $\sum_{\sigma\in K_I} n=n^{|I|+1}$ and $\sum_{\sigma\in K_I}|N_j(\sigma)|=|\Delta_{I\cup\{j\}}^{(t)}|$)}\\
&\leq \delta_{|I|+1} n^{|I|+1} \qquad \text{(by induction)},
\end{align*}
where the equation $\sum_{\sigma\in K_I}|N_j(\sigma)|=|\Delta_{I\cup\{j\}}^{(t)}|$ used at the third equation follows from the fact that 
$N_j(\sigma)$ and $N_j(\sigma')$ are disjoint for distinct $\sigma, \sigma'\in K_I$ and $\Delta_{I\cup\{j\}}^{(t)}=\{\sigma\cup\{v\}: \sigma\in K_I, v\in N_j(\sigma)\}$.
Thus, we obtain
$|\Delta_{I,j}^{(t)}|\geq \left(1-\frac{\delta_{|I|+1}}{t}\right)n^{|I|}$
for any $j\notin I$.
Since $\Delta_I^{(t)}=\bigcap_{j: j\notin I} \Delta_{I,j}^{(t)}$,
we obtain 
$|\Delta_{I}^{(t)}|\geq \left(1-\frac{k\delta_{|I|+1}}{t}\right)n^{|I|}
=\left(1-\delta_{|I|} \right)n^{|I|}$ as required.
\end{proof}

Let $\Delta^{(t)}:=\bigcup_{I\subseteq\{1,\dots, k\}} \Delta_I^{(t)}$.
A vertex set $X$ is said to be {\em $t$-perfect} in $\Delta$ if 
$X$ induces a complete $k$-partite sub-complex in $\Delta$
such that each non-empty face of the sub-complex belongs to $\Delta^{(t)}$.

\begin{lemma}\label{lem:perfect}
Let $\Delta$ be as defined above.
Let $d$ be a positive integer, and $t$ and $\delta_k$ be real numbers satisfying 
\[
t\leq \frac{1}{4(d+1)^{k-1}} \text{ and } \delta_k\leq \frac{1}{4} \left( \frac{t}{k}\right)^{k-1}.
\]
Suppose $|\Delta_{\{1,\dots,k\}}^{(t)}|\geq (1-\delta_k)n^k$.
Then $\Delta$ contains a $t$-perfect vertex set $X$ satisfying $|X\cap V_i|=d$ for all $1\leq i\leq k$.
\end{lemma}
\begin{proof}
Let $V_1, V_2, \dots, V_k$ be the partition of the vertex set of $\Delta$.
We inductively construct $X_i\subseteq V_i$ 
from $i=1$ through $i=k$ in such a way that  $|X_i|=d$ and
$X_1\cup X_2\cup \dots \cup X_i$ forms a perfect set on $V_1\cup V_2\cup \dots V_i$, i.e., every partial transversal on $X_1\cup X_2\cup \dots \cup X_i$ belongs to $\Delta^{(t)}$.
Then $\bigcup_{i=1}^k X_i$ would be a perfect set.

Suppose that we have already constructed such sets $X_1, X_2, \dots, X_{j-1}$, and consider now constructing $X_j\subseteq V_j$.
(Initially, $j=1$.) 
Our task is to find $d$ vertices $v\in V_j$ such that 
every partial transversal in $X_1\cup \dots \cup X_{j-1}\cup\{v\}$ containing $v$ is in $\Delta^{(t)}$.
Equivalently, 
\begin{itemize}
\item[(i)] $\{v\}\in \Delta^{(t)}$, and
\item[(ii)] for every non-empty partial transversal $\sigma$ in $X_1\cup X_2\cup \dots \cup X_{j-1}$, $\sigma+v\in \Delta^{(t)}$.
\end{itemize}
We count the number of vertices in $V_j$ which do not satisfy (i) and (ii).
By Lemma~\ref{lem:good_faces}, the number of vertices which does not satisfy (i) is bounded by $\delta_1 n$,
where $\delta_1$ is as defined in the statement of Lemma~\ref{lem:good_faces}.
As for (ii), for each non-empty partial transversal $\sigma$ in $X_1\cup X_2\cup \dots \cup X_{j-1}$,
$\sigma$ belongs to $\Delta^{(t)}$, so the definition of $\Delta_{I,{j}}^{(t)}$ implies that the number of vertices $v\in V_j$ having $\sigma+v\not\in \Delta^{(t)}$ is bounded by $tn$. Since there are $(d+1)^{j-1}-1$ choices for $\sigma$, the number of vertices in $V_j$ violating (ii) is bounded by $(d+1)^{j-1}tn$.
In total, the number of vertices in $V_j$ violating (i) and (ii) is bounded by $\delta_1 n+(d+1)^{j-1}tn$.

By the definition of $\delta_i$, 
$\delta_i=\left( \frac{k}{t}\right)^{k-i}\delta_k$.
So, if 
\begin{equation}\label{eq:delta_bound}
\left(\left( \frac{k}{t}\right)^{k-1}\delta_k+t(d+1)^{k-1}\right)n\leq n-d,
\end{equation}
then there are at least $d$ vertices in $V_j$ satisfying (i) and (ii), and the desired $X_j$ is obtained by picking any $d$ vertices in $V_j$ satisfying (i) and (ii).
Moreover,  (\ref{eq:delta_bound}) indeed holds because 
$t\leq \frac{1}{4(d+1)^{k-1}}$ and $\delta_k\leq \frac{1}{4} \left( \frac{t}{k}\right)^{k-1}$ by the lemma assumption.
\end{proof}

\begin{proof}[Proof of Lemma~\ref{lem:clique}]
Let 
\[
t=\min\left\{\frac{1}{4(d+1)^{k-1}},\frac{\gamma}{d}\right\} \text{ and } \delta_k=\frac{1}{4} \left( \frac{t}{k}\right)^{k-1}.
\]
By Lemma~\ref{lem:closure_size}, there is a constant $c$ such that 
the closure of $G(n,cn)$ has size at least $(1-\delta_k)n^k$ with probability at least $1-\left(\frac{1}{2k}\right)^{kn}$.
Let $\Delta$ be the simplicial complex corresponding to ${\rm cl}_{P_n}(G(n,cn))$,
and define $\Delta^{(t)}$ as above.
By Lemma~\ref{lem:perfect}, $\Delta$ contains a $t$-perfect vertex set $X$ with $|X\cap V_i|=d$.

Pick $d$ distinct transversals $\sigma_1,\dots, \sigma_d$ on $X\cap V_2, X\cap V_3, \dots, X\cap V_k$.
Since $X$ is perfect, each $\sigma_i$ belongs to $\Delta^{(t)}_{\{2,3,\dots, k\}}$.
In particular, there are at least $(1-t)n$ vertices $v\in V_1$ satisfying $\sigma_i+v\in \Delta^{(t)}_{\{1,2,\dots, k\}}$, or equivalently $\sigma_i\cup\{v\}\in {\rm cl}_{P_n}(G(n,cn))$.
Let $V_1(\sigma_i)=\{v\in V_1: \sigma_i+v\in {\rm cl}_{P_n}(G(n,cn))\}$
and let $V_1^*=(X\cap V_1)\cup \bigcap_{i=1}^d V_1(\sigma_i)$.
Then, for each $v\in V_1^*\setminus X$ and for all $i=1,\dots, d$, ${\rm cl}_{P_n}(G(n,cn))$ contains the hyperedge $\sigma_i\cup\{v\}$.
Since $|V_1(\sigma_i)|\geq (1-t)n$,
we have $|V_1^*|\geq (1-td)n\geq (1-\gamma)n$.

By the symmetry of indices, 
$V_j^*$ can be defined similarly for each $j=1,\dots, k$.
Then, for each $v\in V_j^*\setminus X$,
${\rm cl}_{P_n}(G(n,cn))$ contains $d$ hyperedges consisting of $v$ and vertices in $X$.
Therefore, starting from the balanced complete $k$-partite $k$-graph induced by $X$, one can built a $d$-tree $T$ on $V_1^*\cup V_2^*\cup \dots \cup V_k^*$ by a sequence of degree-$d$ extensions by using only hyperedges in ${\rm cl}_{P_n}(G(n,cn))$.
Since $|V_j^*|\geq (1-td)n\geq (1-\gamma)n$ for all $j$,
the resulting $d$-tree has $(1-\gamma)n$ vertices on each size. 
\end{proof}

\subsection{Proof of Lemma~\ref{lem:extension}}
The aim of this section is to confirm that the counting argument of \cite{lew2023sharp} indeed works even for balanced complete $k$-partite $k$-graphs. 

\begin{proof}[Proof of Lemma~\ref{lem:extension}]
Let $p_-$ and $p_+$ be as defined in (\ref{eq:p}).
For simplicity of notation, 
we denote $G(n,p_-),G(n,p_+), G(n,M_d)$
by $G_-, G_+, G_d$, respectively.
There is a coupling between two random graph models $G(n,p)$ and $G(n,m)$, 
for which a.a.s.~one can force $G_-\subseteq G_d \subseteq G_+$. (See the first paragraph of Section 4 in \cite{lew2023sharp}.) 
Hence, we may assume $G_-\subseteq G_d \subseteq G_+$ in the subsequent discussion.
In particular, $G_-, G_d, G_+$ share the same vertex set, denoted by $V_1\cup V_2 \cup \dots \cup V_k$.

We say that a vertex set $B$ is a {\em not particularly large set (an n.p.l.~set for short)} if $|B\cap V_i|\leq \frac{n}{k}$ for all $i\in \{1,\dots, k\}$.
The main observation is the following claim.

\begin{claim}\label{claim:extension1}
The probability that there is an n.p.l.~vertex set $B$ 
such that 
\begin{itemize}
\item[(i)] $B$ is not $d$-extendable vertex in $G_-$, and 
\item[(ii)] $G_+$ contains an edge $e$ with $e\cap B\cap V_1\neq \emptyset \neq e\cap B\cap V_2$,
\end{itemize}
tends to zero if $n\rightarrow \infty$.
\end{claim}
\begin{proof}
Consider an n.p.l.~vertex set $B$ with $b_i:=|B\cap V_i|\leq \frac{n}{k}$.
For a vertex $v\in V_i$, $K_n^k$ has $\prod_{j: j\neq i} (n-b_j)$ edges contained in $\{v\}\cup \bigcup_{j: j\neq i} (V_j\setminus B)$. Hence, 
the probability that $v$ is not $d$-extensible in $G_-$ is equal to the probability that ${\rm Bin}\left(\prod_{j: j\neq i} (n-b_j),p_-\right)\leq d-1$ holds.
Since $b_i=|B\cap V_i|$, the probability that $B$ has no $d$-extendable vertex is 
\[
\prod_{i=1}^k \left( \mathbb{P}[{\rm Bin}(\prod_{j: j\neq i} (n-b_j),p_-)\leq d-1]\right)^{b_i}.
\]
On the other hand, the number of hyperedges $e$ in $K_n^k$ satisfying 
$e\cap B\cap V_1\neq \emptyset \neq e\cap B\cap V_2$ is 
$b_1b_2n^{k-2}$.
Hence, by Markov's inequality, the probability that 
$B$ satisfies (ii) is bounded by $b_1b_2n^{k-2}p_+$.
In total, the probability that there is an n.p.l.~set $B$ of (i) and (ii) is bounded by 
\begin{equation}\label{eq:prob}
\sum_{b_1=1}^{\frac{n}{k}} 
\dots \sum_{b_k=1}^{\frac{n}{k}} \left( \prod_{i=1}^k \binom{n}{b_i} \right) \left( \prod_{i=1}^k \left( \mathbb{P}[{\rm Bin}(\prod_{j: j\neq i} (n-b_j),p_-)\leq d-1]\right)^{b_i} \cdot \min\{1,b_1b_2n^{k-2}p_+\}\right).
\end{equation}
The remaining discussion is devoted to evaluating (\ref{eq:prob}).

For a given $i\in \{1,\dots,k\}$, we have
\begin{align*}
&\mathbb{P}[{\rm Bin}(\prod_{j:j\neq i}(n-b_j),p_-)\leq d-1]\\
&=\sum_{\ell=0}^{d-1}\binom{\prod_{j:j\neq i}(n-b_j)}{\ell}(p_-)^{\ell}(1-p_-)^{-\ell+\prod_{j:j\neq i}(n-b_j)} \quad (\text{by the definition of binary distribution})\\
&\leq \sum_{\ell=0}^{d-1} n^{(k-1)\ell} (p_-)^\ell (1+o(1))e^{-\left(\prod_{j:j\neq i}(n-b_j)\right)p_-}
\quad (\text{since $1+x\leq e^x$, $\ell\leq d$, and $d$ is constant}).
\end{align*}
Denote $g(n)=\log n+(d-1)\log\log n-\log\log \log n$ for simplicity.
Then $p_-=g(n)/n^{k-1}$. By the Weiastrauss product inequality,
we have 
\[\left(\prod_{j:j\neq i} (n-b_j)\right)p_-=g(n) \left( \prod_{j: j\neq i} \left(1-\frac{b_j}{n}\right) \right)\geq g(n)\left(1-\frac{\sum_{j:j\neq i} b_j}{n} \right).
\]
Combining this with the definition of $g(n)$ and $\sum_{\ell=0}^{d-1}(1+o(1))(\log n)^\ell=(1+o(1))(\log n)^{d-1}$,
one can further simplify the above bound as follows:
\begin{align*}
\mathbb{P}[{\rm Bin}(\prod_{j:j\neq i}(n-b_j),p_-)\leq d-1]
&\leq (1+o(1))(\log n)^{d-1} e^{-g(n)} e^{\frac{g(n)\sum_{j:j\neq i}b_j}{n}} \\
&=(1+o(1))\frac{\log\log n}{n}e^{\frac{g(n)\sum_{j:j\neq i}b_j}{n}}.
\end{align*}
Combining this with $\binom{n}{b_i}\leq \left(\frac{en}{b_i}\right)^{b_i}$, we further obtain
\begin{equation}
\binom{n}{b_i} \left(\mathbb{P}[{\rm Bin}(\prod_{j:j\neq i}(n-b_j),p_-)\leq d-1]\right)^{b_i}
\leq \left( 
\frac{(e+o(1))\log\log n}{b_i} e^{\frac{g(n)\sum_{j:j\neq i}b_j}{n}} 
\right)^{b_i}.
\end{equation}
We next use $\sum_{1\leq i \leq k} (k-1)b_i^2\geq \sum_{1\leq i<j\leq k} 2b_ib_j$ to obtain
\[
\prod_{i=1}^k \left(e^{\frac{g(n)}{n}\sum_{j:j\neq i}b_j}\right)^{b_i}
= e^{\frac{g(n)}{n}\left(\sum_{1\leq i<j \leq k}2b_ib_j\right)}
\leq e^{\frac{g(n)}{n}\left(\sum_{1\leq i \leq k}(k-1)b_i^2\right)}
=\prod_{i=1}^k e^{\frac{(k-1)g(n)}{n}b_i}.
\]
Hence, we get
\begin{equation}\label{eq:prob3}
\prod_{i=1}^k \binom{n}{b_i} \left(\mathbb{P}[{\rm Bin}(\prod_{j:j\neq i}(n-b_j),p_-)\leq d-1]\right)^{b_i}
\leq \prod_{i=1}^k\left( 
\frac{(e+o(1))\log\log n}{b_i} e^{\frac{(k-1)g(n)}{n}b_i} 
\right)^{b_i}.
\end{equation}
Let us define a function $h:\mathbb{R}\rightarrow \mathbb{R}$ in $b$ by
\[
h(b)=\frac{e\log\log n}{b} e^{\frac{(k-1)g(n)}{n}b}. 
\]
By substituting (\ref{eq:prob3}) into (\ref{eq:prob}), in order to complete the proof, it suffices to show that the following term goes to zero when $n\rightarrow \infty$:
\begin{equation}\label{eq:prob4}
\begin{split}
&\sum_{b_1=1}^{\frac{n}{k}} \sum_{b_2=1}^{\frac{n}{k}} \dots \sum_{b_k=1}^{\frac{n}{k}} \left(\prod_{i=1}^k
h(b_i)^{b_i} \cdot \min\{1,b_1b_2n^{k-2}p_+\} \right) \\
&=\prod_{i=3}^k 
\left(\sum_{b_i=1}^{\frac{n}{k}} 
h(b_i)^{b_i} \right)
 \cdot 
\left( \sum_{b_1=1}^{\frac{n}{k}}\sum_{b_2=1}^{\frac{n}{k}} \left(
h(b_1)^{b_1} h(b_2)^{b_2} \cdot \min\{1,b_1b_2n^{k-2}p_+\} \right)\right).
\end{split}
\end{equation}

We next give the bound of $h(b_i)$.
\begin{claim}\label{claim:prob}
For each $i\in \{1,\dots, k\}$,
\begin{align*}
\sum_{b_i=(\log\log n)^2}^{\frac{n}{k}}
h(b_i)^{b_i} 
&\leq \left((1+o(1))\frac{e^2}{\log\log n}\right)^{(\log\log n)^2}
\end{align*}
and
\begin{align*}
\sum_{b_i=1}^{\frac{n}{k}} h(b_i)^{b_i}&
\leq ((e+o(1))\log\log n)^{e^3\log \log n}.
\end{align*}
\end{claim}
\begin{proof}
The function $b\mapsto \frac{e^{-cb}}{b}$ with a positive constant $c$ is convex over $b>0$, and hence the maximum value of $h(b)$ over an interval in the positive orthant is attained at an endpoint of the interval. 
This fact implies that the maximum value of $h(b_i)$ over $b_i\in \{1, 2, \dots, \frac{n}{k}\}$ is attained at $b_i=1$ or $b_i=\frac{n}{k}$.
Using $g(n)=(1+o(1))\log n$, a direct computation shows that 
$h(1)\geq \log \log n$ and 
$h(\frac{n}{k})=n^{-\frac{1}{k}}{\rm polylog}(n)$ if $n$ is sufficiently large,
so for asymptotic analysis we may suppose that the maximum is attained at $b_i=1$.
Hence, by using $g(n)=(1+o(1))\log n$ again,
\begin{equation}\label{eq:prod5}
h(b_i) \leq e \log\log n.
\end{equation}
By the same argument, the function $h(b)$ attains the maximum at $b=(\log\log n)^2$ over the interval $[(\log\log n)^2, \frac{n}{k}]$ if $n$ is sufficiently large.
Hence, we obtain
\begin{equation}\label{eq:prod6}
\begin{split}
\sum_{b_i=(\log\log n)^2}^{\frac{n}{k}}
h(b_i)^{b_i} 
&=
\sum_{b_i=(\log\log n)^2}^{\frac{n}{k}}
\left(\frac{e\log\log n}{b_i} e^{\frac{(k-1)g(n)}{n}b_i}\right)^{b_i} \\
&\leq 
\sum_{b_i=(\log\log n)^2}^{\frac{n}{k}} \left(\frac{e^2}{\log \log n}\right)^{b_i}
=\left((1+o(1))\frac{e^2}{\log\log n}\right)^{(\log\log n)^2},
\end{split}
\end{equation}
implying the first inequality in the statement.

To see the second inequality in the statement, we use the same trick to see
\begin{equation*}
\sum_{b_i=e^3 \log\log n }^{\frac{n}{k}}
h(b_i)^{b_i} 
=
\sum_{b_i=e^3 \log\log n}^{\frac{n}{k}}
\left(\frac{e\log\log n}{b_i} e^{\frac{(k-1)g(n)}{n}b_i}\right)^{b_i}
\leq 
\sum_{b_i=e^3 \log\log n}^{\frac{n}{k}} \left(\frac{1}{e}\right)^{b_i}
\leq 2.
\end{equation*}
Combining this with (\ref{eq:prod5}), we obtain
\begin{equation}\label{eq:prod7}
\sum_{b_i=1}^{\frac{n}{k}} h(b_i)^{b_i}=
\sum_{b_i=1}^{e^3\log\log n-1} h(b_i)^{b_i}
+
\sum_{b_i=e^3\log \log n}^{\frac{n}{k}} h(b_i)^{b_i}
\leq ((e+o(1)) \log\log n)^{e^3\log \log n}
\end{equation}
for each $i$.
\end{proof}

We next evaluate the last term of (\ref{eq:prob4}).
\begin{claim}\label{claim:prob2}
\begin{align*}
&\sum_{b_1=1}^{\frac{n}{k}}\sum_{b_2=1}^{\frac{n}{k}} 
\left(h(b_1)^{b_1} h(b_2)^{b_2} \cdot \min\{1,b_1b_2n^{k-2} p_+\} \right) \\
&\leq (1+o(1))\left(n^{-1}{\rm polylog}(n) + 
2\frac{(e\log \log n)^{e^3\log \log n}e^{2(\log\log n)^2}}{(\log \log n)^{(\log\log n)^2}}\right)
\end{align*}
\end{claim}
\begin{proof}
Recall that $p_+=\frac{ \log n+(d-1)\log \log n+\log \log \log n}{n^{k-1}}$.
Hence, if $b_1, b_2\leq (\log\log n)^2$, 
\begin{equation}\label{eq:prob8}
b_1b_2n^{k-2}p_+\leq \frac{(1+o(1))\log n}{n} \cdot (\log \log n)^4.
\end{equation}
Hence, by (\ref{eq:prob8}) and Claim~\ref{claim:prob},
\begin{align*}
&\sum_{b_1=1}^{\frac{n}{k}}\sum_{b_2=1}^{\frac{n}{k}} 
\left(h(b_1)^{b_1} h(b_2)^{b_2} \cdot \min\{1,b_1b_2n^{k-2} p_+\} \right)\\
&\leq \sum_{b_1=1}^{(\log \log n)^2}\sum_{b_2=1}^{(\log \log n)^2}
\left(h(b_1)^{b_1} h(b_2)^{b_2} \cdot b_1b_2n^{k-2} p_+ \right) \\
 &+\sum_{b_1=(\log \log n)^2}^{\frac{n}{k}}\sum_{b_2=1}^{\frac{n}{k}} 
 h(b_1)^{b_1} h(b_2)^{b_2}
 +\sum_{b_1=1}^{\frac{n}{k}}\sum_{b_2=(\log \log n)^2}^{\frac{n}{k}} 
 h(b_1)^{b_1} h(b_2)^{b_2} \\
& \leq \left( ((e^2+o(1))\log\log n)^{2(\log \log n)^2} \cdot \frac{\log n (\log \log n)^4}{n} + 2(1+o(1)) \frac{(e\log \log n)^{e^3\log \log n}e^{2(\log\log n)^2}}{(\log \log n)^{(\log \log n)^2}} \right) \\
&\leq (1+o(1))\left(n^{-1}{\rm polylog}(n) + 
2\frac{(e\log \log n)^{e^3\log \log n}e^{2(\log\log n)^2}}{(\log \log n)^{(\log\log n)^2}}\right)
\end{align*}
as required.
\end{proof}

We are now ready to complete the proof of Claim~\ref{claim:extension1}.
Our goal is to prove that (\ref{eq:prob}) converges to zero if $n\rightarrow \infty$.
By (\ref{eq:prob3}), (\ref{eq:prob}) is bounded by (\ref{eq:prob4}).
By Claims~\ref{claim:prob} and \ref{claim:prob2}, 
(\ref{eq:prob4}) is bounded by 
\begin{align*}
(1+o(1))(e^2\log\log n)^{(k-2)e^3\log \log n}\cdot 
\left(n^{-1}{\rm polylog}(n) + 
2\frac{(e^2\log \log n)^{e^3\log \log n}e^{2(\log\log n)^2}}{(\log \log n)^{(\log\log n)^2}}\right)
\end{align*}
This converges to zero when $n\rightarrow\infty$,
and the proof of Claim~\ref{claim:extension1} is completed.
\end{proof}

By the symmetry of $V_1, V_2, \dots, V_k$ in $K_n^k$ and the union bound, Claim~\ref{claim:extension1} implies that the probability that there is a n.p.l.~vertex set $B$ such that 
\begin{itemize}
    \item[(i)] $B$ is not $d$-extendable in $G_-$, and
    \item[(ii')] $G_+$ contains an edge $e$ with $e\cap B\cap V_i\neq \emptyset \neq e\cap B\cap V_j$ for some pair of distinct $i,j$ with $1\leq i,j\leq k$
\end{itemize}
tends to zero if $n\rightarrow \infty$.

Now, suppose that $G_d$ has an n.p.l.~vertex set $B$ which is 
non $d$-extendable in $G_d$. 
By $G_-\subset G_d$, $B$ is not $d$-extendable in $G_-$.
Pick a vertex $v\in B\cap V_1$. Then $v$ is not $d$-extendable in $G_d$.
Since $v$ has a degree at least $d$ in $G_d$, $G_d$ has at least one edge $e$ that contains $v$ and other vertices in $B$.
It follows from $G_d\subseteq G_+$ that $G_+$ contains an edge $e$ with 
$e\cap B\cap V_i\neq \emptyset \neq e\cap B\cap V_j$ for some distinct $i,j$.
Consequently, if $G_d$ has a non-$d$-extendable n.p.l.~set $B$, then 
$B$ satisfies (i) and (ii').
The latter probability tends to zero, and hence so does the former.
\end{proof}

\section{Rigidity Threshold}\label{sec:main}
In this section we prove the main result of this paper: probability threshold for the rigidity of random $k$-graphs.
As a warm-up, we shall first discuss on local rigidity and then move to global rigidity.

Throughout this section, we consider the case
when $\bn=(n,\dots, n)$.

\subsection{Threshold for local rigidity}
Consider the generic rigidity matroid ${\cal T}_{n,d}$ defined in Section~\ref{sec:pre}, which is the matroid on $E(K_n^k)$ defined 
by the row independence of $J\sigma^d_{\bn}(\bp)$ at a generic $d$-dimensional point configuration $\bp$.
As we noted in Section~\ref{sec:pre}, $G\subset K_{n}^k$ is locally rigid in $\mathbb{F}^d$ if and only if $E(G)$ has rank $dkn-d(k-1)$.
(Note that $N=kn$ throughout this section since $\bn=(n,\dots,n)$.)
So, in order to show that the random graph $G(n,M_d)$ is locally rigid in $\mathbb{F}^d$, it suffices to show that 
$G(n,M_d)$ has rank $dkn-d(k-1)$ in ${\cal T}_{n,d}$.
The latter claim can be shown by applying Theorem~\ref{thm:random}.
To see this, we need to check the local rigidity of $k$-partite $d$-trees. 

\begin{lemma}\label{lem:local_rigidity}
Let $d,k$ be positive integers with $k\geq 3$.
Any $k$-partite $d$-tree is locally rigid in $\mathbb{F}^d$.
\end{lemma}
\begin{proof}
Recall that any $k$-partite $d$-tree is constructed from $K_d^k$ by a sequence of degree-$d$ extensions.
Hence we need to show 
(1) $K_d^k$ is locally rigid in $\mathbb{F}^d$ 
and (2) any degree-$d$ extension preserves local rigidity in $\mathbb{F}^d$.

(1) follows from a known fact.
Indeed, it is known (and can be easily shown by an elementary linear algebraic argument) 
that the dimension of the $d$-secant of the affine cone of the Segre variety $\overline{{\rm im} \ \sigma_{\bn}}$ is 
$dkn-d(k-1)$ if $\bn=(n,\dots, n)$ and $n\geq d$.
(This is not a sharp result. See, e.g.~\cite{abo2009,bernardi} for stronger results.)
Since this dimension is equal to 
the rank of $Jf_{K_{\bn}^k}(\bp)$ at a generic $\bp$, 
the local rigidity of $K_d^k$ follows from Proposition~\ref{prop:infinitesimal}.

(2) is also a known fact.
Indeed, in \cite[Corollary 6.4]{cruickshank2023identifiability}, it has been show that, if $G$ is a locally rigid $k$-graph in $\mathbb{F}^d$ and $H$ is obtained from $G$ by a degree-$d$ extension, then $H$ is locally rigid in $\mathbb{F}^d$.
\end{proof}

We are now ready to solve the local rigidity problem.

\begin{theorem}\label{thm:local_rigidity}
Let $d,k$ be positive integers with $k\geq 3$.
Then, a.a.s.~$G(n,M_d)\subset K_n^k$ is locally rigid in $\mathbb{F}^d$.
\end{theorem}
\begin{proof}
We apply Theorem~\ref{thm:random}
to the generic rigidity matroid ${\cal T}_{n,d}$ on $E(K_n^k)$.
Theorem~\ref{thm:random} implies that
the graph of the closure of $G(n,M_d)$ in ${\cal T}_{n,d}$
contains a spanning $k$-partite $d$-tree with high probability if $n$ is sufficiently large.
By Lemma~\ref{lem:local_rigidity},
the spanning $k$-partite $d$-tree is locally rigid in $\mathbb{F}^d$. 
So the rank of the edge set of $G(n,M_d)$ in ${\cal T}_{n,d}$ is $dkn-d(k-1)$,
which means that $G(n,M_d)$ is locally rigid in $\mathbb{F}^d$.
\end{proof}

Note that any locally rigid $k$-graph $G$ in $\mathbb{F}^d$ has minimum degree at least $d$.
Hence the statement in Theorem~\ref{thm:local_rigidity} is best possible.
By Proposition~\ref{prop:min_degree}, 
Theorem~\ref{thm:local_rigidity} also implies that 
$p=\frac{\log n+(d-1)\log \log n+o(\log \log n)}{n^{k-1}}$ is the probability threshold for the local rigidity of $G(n,p)$ in $\mathbb{F}^d$.

\subsection{Threshold for 1-dimensional global rigidity}
Next, we move to establishing the global rigidity threshold.
We first solve the 1-dimensional case in this subsection.

Our idea is to verify the condition of Theorem~\ref{thm:1d_global_real} for $\mathbb{F}=\mathbb{R}$
and that of Theorem~\ref{thm:1d_global_complex} for $\mathbb{F}=\mathbb{C}$ in the random graph $G(n,M_1)$.
As in the last subsection, this will be done by applying the observations in Section~\ref{sec:random}.
We first remark on the following easy observation.
\begin{prop}\label{prop:1d_extension}
Over any field $\mathbb{F}$, 
any $k$-partite $1$-tree $G$ with $N$ vertices satisfies 
$\rank_{\mathbb{F}} I_G=N-(k-1)$.
\end{prop}
\begin{proof}
We first claim that the degree-one extension preserves the dimension of the kernel of the incidence matrix, i.e.,
if $H$ is obtained from $G$ by a degree-one extension, then 
$\rank_{\mathbb{F}} I_H=\rank_{\mathbb{F}} I_G+1$.
This fact can be checked trivially 
 by observing that 
$I_H$ is written as 
$I_H=\begin{pmatrix} I_G & \ast \\ {\bf 0} & 1\end{pmatrix}$
if $H$ is obtained from $G$ by a degree-one extension.

Now suppose $G$ is a $k$-partite $1$-tree.
Then $G$ is constructed from $K_1^k$ by a sequence of degree-one extensions. 
When $G=K_1^k$, $G$ has only one edge
and hence $\rank_{\mathbb{F}} I_G=1=N-(k-1)$ as required.
When $G\neq K_1^k$, the claim follows by applying
degree-one extensions inductively. 
\end{proof}

\begin{theorem}\label{thm:1d_global}
Let $k$ be a positive integer with $k\geq 3$
and $\mathbb{F}\in \{\mathbb{R},\mathbb{C}\}$.
Then, a.a.s.~$G(n,M_1)\subset K_n^k$ is globally rigid in $\mathbb{F}^1$.
\end{theorem}
\begin{proof}
We first solve the case when $\mathbb{F}=\mathbb{R}$.
By Proposition~\ref{prop:1d_extension} and Theorem~\ref{thm:random}, it holds
a.a.s.~that $\rank_{\mathbb{R}} I_{G(n,M_1)}=\rank_{GF(2)} I_{G(n,M_1)}=kn-(k-1)$.
Then the 1-dimensional global rigidity of $G(n,M_1)$ over $\mathbb{R}$ follows from Theorem~\ref{thm:1d_global_real}.

We next consider the case when $\mathbb{F}=\mathbb{C}$.
Let $q$ be a prime 
and let ${\cal I}_q$ be the linear matroid represented by the incidence matrix $I_{K_n^k}$ over $GF(q)$. Since each column of $I_{K_n^k}$ is indexed by $E(K_n^k)$, we may regard ${\cal I}_q$ as a matroid on the ground set $E(K_n^k)$. 
Let ${\rm cl}_{q}(H)$ be the graph of the closure of $H$
in ${\cal I}_q$ for any $H\subseteq K_n^k$.
By Lemma~\ref{lem:clique} applied to ${\cal I}_q$,
there is a constant $c$ (independent of $n$ and $q$) such that,
with probability at least $1-\left( \frac{1}{2k}\right)^{kn}$,
${\rm cl}_{q}(G(n,cn))$ contains a $k$-partite $1$-tree whose vertex set has size at least $(1-\frac{1}{k})n$ on each side.

In order to apply Theorem~\ref{thm:1d_global_complex}, we need to handle ${\rm cl}_{q}(G(n,cn))$ for all prime $q$ satisfying 
$q\leq k^{kn/2}$ simultaneously.
This can be done by union bound.
Indeed, by union bound, with probability at least $1-\left( \frac{1}{2}\right)^{kn}$,
${\rm cl}_{q}(G(n,cn))$ contains a $k$-partite $1$-tree $T_q$ whose vertex set has size at least $(1-\frac{1}{k})n$ for every prime $q$ satisfying $q\leq k^{kn/2}$.
Then, Lemma~\ref{lem:extension} says that a.a.s.~$G(n,M_1)$ has the property that every vertex set $B$ with $|B\cap V_i|\leq \frac{n}{k}$ is $1$-extendable for all $i=1,\dots, k$.
Note that $M_1\gg cn$ if $n$ is sufficiently large by Proposition~\ref{prop:min_degree}.
Hence, if $n$ is sufficiently large, 
$T_q$ can be extended to a spanning $k$-partite $1$-tree for each prime $q$.
By Proposition~\ref{prop:1d_extension}, 
this implies that
a.a.s.~$\rank_{\mathbb{R}} I_{G(n,M_1)}=\rank_{GF(q)} I_{G(n,M_1)}=kn-(k-1)$ 
for any prime $q$ satisfying $q\leq k^{kn/2}$.
We can now apply Theorem~\ref{thm:1d_global_complex} to conclude that
a.a.s.~$G(n,M_1)$ is globally rigid in $\mathbb{C}^1$.
\end{proof}

The following statement was implicit in the last proof.
Since incidence matrices are fundamental objects in combinatorics, we shall state it as an independent theorem.
\begin{theorem}\label{thm:incidence}
Let $k$ be a positive integer with $k\geq 3$
and $\mathbb{F}$ be any field.
Then, a.a.s.~$G(n,M_1)\subset K_n^k$ satisfies
$\rank_{\mathbb{F}} I_{G(n,M_1)}=kn-(k-1)$
\end{theorem}

An analogous result for the random subgraph of a complete $k$-graph is given in \cite{cooper19}.

\subsection{Polymatroids for certifying \texorpdfstring{$d$}{d}-tangential weak non-defectiveness}
Our final goal is to prove the global rigidity counterpart of Theorem~\ref{thm:local_rigidity}.
Extending the proof of Theorem~\ref{thm:local_rigidity} for local rigidity to global rigidity requires a new idea. A key in the proof of Theorem~\ref{thm:local_rigidity} is the fact that locally rigidity is characterized in terms of the underlying matroid
and hence the local rigidity problem can be captured in the setting of Section~\ref{sec:random}. 
On the other hand, global rigidity is not a matroidal property and there seems no obvious underlying matroid that can be used for certifying global rigidity. 
We will overcome this difficulty by introducing a new polymatroid that certifies $d$-tangential weak non-defectiveness.

For simplicity of description, we shall first look at the condition (iii) of Theorem~\ref{thm:MM_test}.
(Notice that, in view of Theorem~\ref{thm:MM_test}, testing this condition is the only missing piece for certifying the $d$-dimensional global rigidity of $G(n,M_{d+1})$.)
Our tentative goal is to define a polymatroid whose rank certifies the condition (iii) of Theorem~\ref{thm:MM_test}.

We first introduce a canonical basis in the cycle space.
Consider $K_{\bn}^k$ with $\bn=(n_1,\dots, n_k)$ and $N=\sum_i n_i$.
Recall that the rank of the incidence matrix $I_{K_{\bn}^k}$ of $K_{\bn}^k$ is $N-(k-1)$.
Suppose that a subgraph $G_0$ of $K_{\bn}^k$
satisfies 
\[
\rank I_{G_0}=|E(G_0)|=N-(k-1),
\]
i.e.,
the set of column vectors of $I_{K_{\bn}^k}$ indexed by the edges of $G_0$ forms a basis of the column space of $I_{K_{\bn}^k}$.
Then, for each $e\in E(K_{\bn}^k)\setminus E(G_0)$,
the kernel of $I_{G_0+e}$ is one-dimensional
and thus a non-zero $\omega_e\in \ker I_{G_0+e}$ is uniquely defined up to scaling. 
Since $G_0\subseteq K_{\bn}^k$, we may regard $\omega_e$ as an element of $\ker I_{K_{\bn}^k}$ by setting $\omega_e(f)=0$ for all $f\in E(K_{\bn}^k)\setminus (E(G_0+e))$.

\begin{prop}\label{prop:canonical_basis}
Let $G_0$ be a $k$-partite $k$-graph with $G_0\subseteq K_{\bn}^k$ and $\rank I_{G_0}=|E(G_0)|=N-(k-1)$.
Then, $\{\omega_e: e\in E(K_{\bn}^k)\setminus E(G_0)\}$ forms a basis of $\ker I_{K_{\bn}^k}$.
\end{prop}
\begin{proof}
$\{\omega_e: e\in E(K_{\bn}^k)\setminus E(G_0)\}$ is a linearly independent set
because, for $e,e'\in E(K_{\bn}^k)\setminus E(G_0)$, 
$\omega_e(e')\neq 0$ if and only if $e=e'$.
Moreover, $\{\omega_e: e\in E(K_{\bn}^k)\setminus E(G_0)\}$ spans $\ker I_{K_{\bn}^k}$ because $\dim \ker I_{K_{\bn}^k}=|E(K_{\bn}^k)\setminus E(G_0)|$ holds by the assumption that the columns of $E(G_0)$ form a basis of the column space of $I_{K_{\bn}^k}$.
\end{proof}

For simplicity of description, we shall also set $\omega_e$ to be the zero vector if $e\in E(G_0)$. 
Then, $\{\omega_e: e\in E(K_{\bn}^k)\}$ is called a {\em canonical cycle basis} with respect to $G_0$.

Using the canonical cycle basis $\{\omega_e: e\in E(K_{\bn}^k)\}$, we define a linear polymatroid as follows.
We define a set function $r_{G_0}:2^{E(K_{\bn}^k)}\rightarrow \mathbb{Z}$ by
\[
r_{G_0}(F)=\dim \langle {\rm im}\ A_{\omega_e}: e\in F\rangle\qquad (F\subseteq E(K_{\bn}^k)),
\]
where $\langle \cdot \rangle$ denotes the linear span.
Then, $P_{G_0}:=(E(K_{\bn}^k),r_{G_0})$ is a polymatroid because 
it can be linearly represented by associating the linear space 
${\rm im}\ A_{\omega_e}$ in $\mathbb{F}^N$ with each $e\in E(K_{\bn}^k)$. 
This polymatroid $P_{G_0}$ is called the {\em cycle polymatroid} with respect to $G_0$. 
Note that the cycle polymatroid has rank at most $N-k$
by Proposition~\ref{prop:lowerbound}.

\begin{prop}\label{prop:cycle_polymatroid}
Let $G_0\subseteq K_{\bn}^k$ be a $k$-partite $k$-graph with $N$ vertices and $\rank I_{G_0}=|E(G_0)|=N-(k-1)$, 
and let $(E(K_{\bn}^k),r_{G_0})$ be the cycle polymatroid with respect to $G_0$. Then, for any $G$ with $G_0\subseteq G\subseteq K_{\bn}^k$,
\[
r_{G_0}(E(G))=\dim \left< {\rm im } A_{\omega}: \omega\in \ker I_G \right>.
\]
In particular, $G$ satisfies (\ref{eq:simplified_kernel}) if $r_{G_0}(E(G))=N-k$.
\end{prop}
\begin{proof}
Let $\{\omega_e: e\in E(G)\}$ be a canonical cycle basis with respect to $G_0$.
By Proposition~\ref{prop:canonical_basis},
any $\omega\in \ker I_G$ is written as a linear combination of $\omega_e\ (e\in E(G))$.
Since each entry of $A_{\omega}$ is linear in the entries of $\omega$, $A_{\omega}$ is also written as a linear combination of $A_{\omega_e}$ over $e\in E(G)$.
This in turn implies that 
${\rm im}\ A_{\omega}\subseteq \langle {\rm im}\ A_{\omega_e}: e\in E(G)\rangle$
and hence 
$r_{G_0}(E(G))=\dim \langle {\rm im}\ A_{\omega_e}: e\in E(G)\rangle=\dim \langle {\rm im}\ A_{\omega}: \omega\in \ker I_G \rangle$.
\end{proof}

Exactly same argument applies to the condition (\ref{eq:kernel}) for $d$-tangential weak non-defectiveness just
by replacing $I_{G}$ with $Jf_{G}^d(\bp)^{\top}$ 
and $A_{\omega}$ with $A_{\omega}^1$.
The corresponding linear polymatroid is called the {\em $d$-dimensional stress polymatroid} with respect to $G_0$.
(Note that, like the generic rigidity matroid, the stress polymatroid is independent of the choice of $\bp$ as long as $\bp$ is generic.)
The proof of Proposition~\ref{prop:cycle_polymatroid} can be easily adapted to show the following.
\begin{prop}\label{prop:stress_polymatroid}
Let $G_0\subseteq K_{\bn}^k$ be a $k$-partite $k$-graph with 
$N$ vertices and $\rank Jf_{G_0}^d(\bp)=|E(G_0)|=dN-d(k-1)$, 
and let $(E(K_{\bn}^k),r_d)$ be the $d$-dimensional stress polymatroid with respect to $G_0$. Then, for any $G$ with $G_0\subseteq G\subseteq K_{\bn}^k$,
\[
r_d(E(G))=\dim \left< {\rm im } A_{\omega}^1: \omega\in \ker Jf_{G_0}^d(\bp)^{\top} \right>.
\]
In particular, $G$ satisfies (\ref{eq:kernel}) if $r_d(E(G))=N-k$.
\end{prop}

\subsection{Threshold for \texorpdfstring{$d$}{d}-dimensional global rigidity}
\label{subsec:main_proof}
We begin with a lemma about $k$-partite $d$-trees.
\begin{lemma}\label{lem:tree_cycle_space}
Any $k$-partite $2$-tree satisfies (\ref{eq:simplified_kernel}).
\end{lemma}
\begin{proof}
In Example~\ref{ex:K}, we have seen that $K_{(2,2,1,\dots, 1)}^k$ satisfies (\ref{eq:simplified_kernel}).
A spanning subgraph of $K_2^k$ 
can be constructed from $K_{(2,2,1,\dots, 1)}^k$
by degree-two extensions 
and then any $k$-partite $2$-tree can be constructed from $K_2^k$
by degree-two extensions.
So it suffices to show that each degree-two extensions
from preserves the property of having (\ref{eq:simplified_kernel}).

Suppose $G$ satisfies (\ref{eq:simplified_kernel})
and $H$ is obtained from $G$ by a degree-two extension by adding a new vertex $v$ and new two edges $e_1, e_2$ containing $v$.
Since $G\subset H$, there is an obvious linear injection $\psi$ from $\ker I_G$ to $\ker I_H$ that sends $\omega\in \ker I_G$ 
to $\omega'\in \ker I_H$ with $\omega'(f)=\omega(f)$ if $f\in E(G)$ and otherwise $\omega'(f)=0$.
Note that $\rank I_H=\rank I_G+1$ and $|E(H)|=|E(G)|+2$.
So $\ker I_H$ has an element $\omega^*$ such that 
$\ker I_H = \langle \psi(\ker I_G)\cup \{\omega^*\}\rangle$.
Since only the columns of $e_1$ and $e_2$ have non-zero entries on the row of $v$ in $I_H$, we actually have $\omega^*(e_1)=-\omega^*(e_2)\neq 0$.

Observe that $A_{\psi(\omega)}=\begin{pmatrix} 0 & 0 \\ 0 & A_{\omega}\end{pmatrix}$ for $\omega\in \ker I_H$
(where the first row and the first column are indexed by the new vertex $v$.)
So, by the fact that $G$ satisfies (\ref{eq:simplified_kernel}), we have
\begin{equation}
\label{eq:cycle_space_extension}
\dim \bigcap_{\omega\in \ker I_G} \ker A_{\psi(\omega)}=k+1,
\end{equation}
and the characteristic vector $\chi_v$ of $v$ is in this common kernel.

Suppose $H$ does not satisfy (\ref{eq:simplified_kernel}).
Since $\psi(\omega)\in \ker I_H$, 
(\ref{eq:cycle_space_extension}) implies that 
$\bigcap_{\omega\in \ker I_G} \ker A_{\psi(\omega)}=\bigcap_{\omega'\in \ker I_H} \ker A_{\omega'}$.
This in particular implies that 
$\chi_v\in \ker A_{\omega^*}$.
However, this leads to a contradiction because 
$\omega^*(e_1)\neq 0$ and $A_{\omega^*}$ has a non-zero entry on the column indexed by $v$.
Therefore, $H$ satisfies (\ref{eq:simplified_kernel})
and any $k$-partite 2-tree satisfies (\ref{eq:simplified_kernel}).
\end{proof}

We are now ready to prove our main theorem.
\begin{theorem}~\label{thm:global}
Let $k,d$ be positive integers with $k\geq 3$ and 
$\mathbb{F}\in \{\mathbb{R},\mathbb{C}\}$.
Then, a.a.s.~$G(n,M_{d+1})\subseteq K_n^k$ is globally rigid in $\mathbb{F}^d$.
\end{theorem}
\begin{proof}
We use Theorem~\ref{thm:MM_test} to check global rigidity in $\mathbb{F}^d$.
By Theorem~\ref{thm:local_rigidity} and Theorem~\ref{thm:1d_global},
a.a.s.~$G(n,M_{d+1})$ satisfies the conditions (i) and (ii) of Theorem~\ref{thm:MM_test}.

It remains to check the condition (iii) of Theorem~\ref{thm:MM_test}.
By Theorem~\ref{thm:incidence}, a.a.s.~$G(n,M_{d+1})$ satisfies 
$\rank I_{G(n,M_{d+1})}=kn-(k-1)$.
Conditioned on this property in $G(n,M_{d+1})$,
$G(n,M_{d+1})$ has a spanning subgraph $G_0$ satisfying 
$\rank I_{G_0}=|E(G_0)|=kn-(k-1)$.
For each such $G_0$, the cycle polymatroid $P_{G_0}=(E(K_n^k),r_{G_0})$ with respect to $G_0$ is defined.
By Theorem~\ref{thm:random},
the graph of the closure of $G(n,M_{d+1})$ in $P_{G_0}$ contains a spanning $k$-partite $2$-tree $T$.
By Lemma~\ref{lem:tree_cycle_space}, 
$T$ satisfies (\ref{eq:simplified_kernel}).

Let $\{\omega_e: e\in E(K_n^k)\}$ be a canonical cycle basis with respect to $G_0$.
Since $E(G_0)$ and $E(T)$ are in the closure of $E(G(n,M_{d+1}))$ in $P_{G_0}$, 
we have 
\begin{align*}
r_{G_0}(E(G(n,M_{d+1})))&\geq r_{G_0}(E(G_0\cup T)) & \\
&=\dim \langle {\rm im} A_{\omega_e}: e\in E(G_0\cup T)\rangle & (\text{by the definition of $r_{G_0}$})\\
&=\dim \langle {\rm im} A_{\omega}: \omega \in \ker I_{G_0\cup T}\rangle & (\text{by Proposition~\ref{prop:cycle_polymatroid}})\\
&\geq \dim \langle {\rm im} A_{\omega}: \omega \in \ker I_{T}\rangle \\
&=kn-k & (\text{since $T$ satisfies (\ref{eq:simplified_kernel})}).
\end{align*}
Since the rank of $P_{G_0}$ is upper bounded by $kn-k$, this implies that $r_{G_0}(E(G(n,M_{d+1})))=dk-k$.
By Proposition~\ref{prop:cycle_polymatroid},
$G(n,M_{d+1})$ satisfies (\ref{eq:simplified_kernel}).

After all, a.a.s.~$G(n,M_{d+1})$ satisfies all the conditions of Theorem~\ref{thm:MM_test} and the $d$-dimensional global rigidity of $G(n,M_{d+1})$ holds.
\end{proof}

By Proposition~\ref{prop:min_degree}, 
Theorem~\ref{thm:local_rigidity} also implies that 
$p\geq \frac{\log n+d\log \log n+o(\log \log n)}{n^{k-1}}$ is a sufficient probability for the global rigidity of $G(n,p)$ in $\mathbb{F}^d$.
This gives the right bound for the leading term $\log n$ including its coefficient.

We should point out that there is a little gap in the coefficient of the second leading term, $\log \log n$.
Specifically, the current upper bound from Theorem~\ref{thm:global} gives $d \log \log n$ for the second leading term.
On the other hand, the minimum degree bound for globally rigid graphs is $d$ and it only gives $(d-1) \log \log n$ for the lower bound of the second leading term.
Theorem~\ref{thm:1d_global} for 1-dimensional global rigidity suggests that $d-1$ would be a right coefficient. 

\subsection{Threshold for \texorpdfstring{$d$}{d}-tangential weak non-defectiveness}
The same argument can be used to check 
$d$-tangential weak non-defectiveness.

The following lemma can be proved in the same manner as Lemma~\ref{lem:tree_cycle_space}.
We omit the details since it is tedious and is not related to the main part of the paper.
\begin{lemma}\label{lem:d+1tree}
Let $G$ be a $k$-partite $(d+1)$-tree.
Then a generic $d$-dimensional framework $(G,\bp)$ of $G$
satisfies (\ref{eq:kernel}).
\end{lemma}

\begin{theorem}~\label{thm:tangential}
Let $k,d$ be positive integers with $k\geq 3$.
Then, a.a.s.~$\overline{\pi_{G(n,M_{d+1})}({\rm im} \ \sigma_{\bn})}$ is $d$-tangentially weakly non-defective over $\mathbb{C}$.
\end{theorem}
\begin{proof}
We check (\ref{eq:kernel}) for a generic $d$-dimensional framework of $G(n,M_{d+1})$.
The proof is identical to that of Theorem~\ref{thm:global}.
We just need to look at 
the $d$-dimensional stress polymatroid 
instead of the cycle polymatroid
and replace
Lemma~\ref{lem:tree_cycle_space} with Lemma~\ref{lem:d+1tree}
and 
Proposition~\ref{prop:cycle_polymatroid} with Proposition~\ref{prop:stress_polymatroid}. 
\end{proof}

Theorem~\ref{thm:tangential} combined with Theorem~\ref{thm:COtest} and Theorem~\ref{thm:1d_global_complex} leads to an alternative proof of Theorem~\ref{thm:MM_test} without relying on Masaratti-Mella's theorem.

\section{Concluding Remarks}
Our proof of the main theorem is done by showing (cf.~Proposition~\ref{prop:iden_global}):
\begin{itemize}[leftmargin=1.8cm]
    \item [Part (i):] the 1-dimensional global rigidity, and  
    \item[Part (ii):] the $d$-identifiability of the underlying variety obtained from the Segre variety by the projection to a random axis-parallel linear subspace.
\end{itemize}
The proof of Part (ii) can be adapted to other varieties,
such as Veronese or Grassmannian, 
to give an asymptotically tight dimension for a random axis-parallel projection to preserve $d$-identifiability.
The corresponding rigidity model is obtained by looking at 
the Veronese map or the Pl{\"u}cker embedding instead of the Segre map,
and the underlying hypergraph would be 
a random subgraph of a complete $k$-uniform (multi or simple) hypergraph.
For complete $k$-uniform (multi or simple) hypergraphs, the discussion in Section~\ref{sec:random} can be adapted to show an analogous statement to Theorem~\ref{thm:random}. (This has been verified in the bachelor project of the first author~\cite{hamaguchi}.)
On the other hand, for Segre-Veronese varieties, the underlying hypergraphs would be more complicated and we do not know if Theorem~\ref{thm:random} can be extended to this case.
In view of this, a reasonable research question would be to extend  Theorem~\ref{thm:random} to a more general random subgraph model.

A tensor $T$ is said to be {\em symmetric} if 
the entries are invariant under any permutation of indices.
It is known that a symmetric tensor $T$ of size $N$ admits a decomposition in the form $T=\sum_{i=1}^d x_i\otimes \dots \otimes x_i$ for some $x_i\in \mathbb{F}^N \ (i=1,\dots, d)$
and the smallest possible integer $d$ is called the {\em symmetric rank} of $T$.
As in the low-rank tensor completions, one can formulate the symmetric low-rank symmetric tensor completion problem
and ask if a completion is determined uniquely.
The unique recovery problem can again be formulated as a hypergraph rigidity,
where the underlying varieties would be secants of Veronese varieties.
See~\cite{cruickshank2023identifiability} for more details on this rigidity model.
The first author has confirmed in his bachelor project~\cite{hamaguchi} that the current proof technique can be applied to derive the symmetric tensor version of Theorem~\ref{thm:main_tensor_completion}. 

There is also a notion of {\em skew-symmetric} tensors, where the underlying varieties would be secants of Grassmannian varieties, see, e.g.,~\cite{bernardi}. It turns out that solving Part (i) is challenging in this case, and it still remains open to identify sampling complexity in the skew-symmetric tensor completion problem.

\subsection*{Acknowledgements}
S.T. was partially supported by JST ERATO
Grant Number JPMJER1903 and JST PRESTO Grant Number JPMJPR2126.
S.T.~would like to thank James Cruickshank, Fatemeh Mohammadi, Anthony Nixon, and Kota Nakagawa for stimulating discussion on the $g$-rigidity of hypergraphs.

\smallskip
\bibliographystyle{abbrv} 
\bibliography{tensor.bib}

\begin{thebibliography}{10}

\bibitem{asimow}
L.~Asimow and B.~Roth.
\newblock The rigidity of graphs.
\newblock {\em Transactions of the American Mathematical Society}, 245:279--289, 1978.

\bibitem{barak}
B.~Barak and A.~Moitra.
\newblock Noisy tensor completion via the sum-of-squares hierarchy.
\newblock In {\em Proceedings of Conference on Learning Theory}, 2016.

\bibitem{bernardi}
A.~Bernardi, E.~Carlini, M.~V. Catalisano, A.~Gimigliano, and A.~Oneto.
\newblock The hitchhiker guide to: Secant varieties and tensor decomposition.
\newblock {\em Mathematics}, 6(12), 2018.

\bibitem{BCC}
C.~Bocci, L.~Chiantini, and G.~Ottaviani.
\newblock Refined methods for the identifiability of tensors.
\newblock {\em Annali di Matematica Pura ed Applicata}, 193:1691--1702, 2014.

\bibitem{breiding2021algebraic}
P.~Breiding, F.~Gesmundo, M.~Micha{\l}ek, and N.~Vannieuwenhoven.
\newblock Algebraic compressed sensing.
\newblock {\em Applied and Computational Harmonic Analysis}, 65:374--406, 2023.

\bibitem{CO2012}
L.~Chiantini and G.~Ottaviani.
\newblock On generic identifiability of 3-tensors of small rank.
\newblock {\em SIAM Journal on Matrix Analysis and Applications}, 33(3):1018--1037, 2012.

\bibitem{COV2014}
L.~Chiantini, G.~Ottaviani, and N.~Vannieuwenhoven.
\newblock An algorithm for generic and low-rank specific identifiability of complex tensors.
\newblock {\em SIAM Journal on Matrix Analysis and Applications}, 35(4):1265--1287, 2014.

\bibitem{chiantini2017generic}
L.~Chiantini, G.~Ottaviani, and N.~Vannieuwenhoven.
\newblock On generic identifiability of symmetric tensors of subgeneric rank.
\newblock {\em Transactions of the American Mathematical Society}, 369(6):4021--4042, 2017.

\bibitem{cooper19}
C.~Cooper, A.~Frieze, and W.~Pegden.
\newblock On the rank of a random binary matrix.
\newblock In {\em Proceedings of the Thirtieth Annual ACM-SIAM Symposium on Discrete Algorithms}, pages 946--955. SIAM, 2019.

\bibitem{cruickshank2023identifiability}
J.~Cruickshank, F.~Mohammadi, A.~Nixon, and S.-i. Tanigawa.
\newblock Identifiability of points and rigidity of hypergraphs under algebraic constraints.
\newblock {\em arXiv preprint arXiv:2305.18990}, 2023.

\bibitem{freeze}
A.~Frieze and M.~K. Pennsylvania.
\newblock {\em Introdution to Random Graphs}.
\newblock Cambridge University Press, 2015.

\bibitem{ghadermarzy}
N.~Ghadermarzy, Y.~Plan, and {\"O}.~Yilmaz.
\newblock Near-optimal sample complexity for convex tensor completion.
\newblock {\em Information and Inference: A Journal of the IMA}, 8(3):577--619, 2019.

\bibitem{GHT10}
S.~J. Gortler, A.~D. Healy, and D.~P. Thurston.
\newblock Characterizing generic global rigidity.
\newblock {\em American Journal of Mathematics}, 132(4):897--939, 2010.

\bibitem{Gortler2014}
S.~J. Gortler and D.~P. Thurston.
\newblock Generic global rigidity in complex and pseudo-{E}uclidean spaces.
\newblock {\em Fields Institute Communications, Rigidity and Symmetry}, pages 131--154, 2014.

\bibitem{hamaguchi}
H.~Hamaguchi.
\newblock Conditions for the unique identifiability of point configurations.
\newblock Bachelor thesis (in {J}apanese), University of Tokyo, 2024.

\bibitem{karnik2024}
C.~Haselby, M.~Iwen, S.~Karnik, and R.~Wang.
\newblock Tensor deli: Tensor completion for low {CP}-rank tensors via random sampling.
\newblock {\em arXiv:2403.09932}, 2024.

\bibitem{abo2009}
A.~Hirotachi, O.~Giorgio, and C.~Peterson.
\newblock Induction for secant varieties of segre varieties.
\newblock {\em Transactions of the American Mathematical Society}, 361(2):767--792, 2009.

\bibitem{JJTunique}
B.~Jackson, T.~Jord\'an, and S.~Tanigawa.
\newblock Unique low rank completability of partially filled matrices.
\newblock {\em Journal of Combinatorial Theory, Series B}, 121:432--462, 2016.

\bibitem{jain2014provable}
P.~Jain and S.~Oh.
\newblock Provable tensor factorization with missing data.
\newblock {\em Advances in Neural Information Processing Systems}, 27, 2014.

\bibitem{keshavan2010matrix}
R.~H. Keshavan, A.~Montanari, and S.~Oh.
\newblock Matrix completion from a few entries.
\newblock {\em IEEE Transactions on Information Theory}, 56(6):2980--2998, 2010.

\bibitem{kiraly2015algebraic}
F.~J. Kir{\'a}ly, L.~Theran, and R.~Tomioka.
\newblock The algebraic combinatorial approach for low-rank matrix completion.
\newblock {\em J. Mach. Learn. Res.}, 16(1):1391--1436, 2015.

\bibitem{Lansberg}
J.~M. Landsberg.
\newblock {\em Tensors: Geoemtry and Applications}.
\newblock AMS, 2012.

\bibitem{lew2023sharp}
A.~Lew, E.~Nevo, Y.~Peled, and O.~E. Raz.
\newblock Sharp threshold for rigidity of random graphs.
\newblock {\em Bulletin of the London Mathematical Society}, 55(1):490--501, 2023.

\bibitem{liu2020}
A.~Liu and A.~Moitra.
\newblock Tensor completion made practical.
\newblock In {\em Advances in Neural Information Processing Systems}, volume~33, pages 18905--18916, 2020.

\bibitem{MMident}
A.~Massarenti and M.~Mella.
\newblock Bronowski's conjecture and the identifiability of projective varieties.
\newblock {\em arXiv preprint arXiv:2210.13524}, 2022.

\bibitem{mu2014square}
C.~Mu, B.~Huang, J.~Wright, and D.~Goldfarb.
\newblock Square deal: Lower bounds and improved relaxations for tensor recovery.
\newblock In {\em International Conference on Machine Learning}, pages 73--81. PMLR, 2014.

\bibitem{potechin2017exact}
A.~Potechin and D.~Steurer.
\newblock Exact tensor completion with sum-of-squares.
\newblock In {\em Conference on Learning Theory}, pages 1619--1673. PMLR, 2017.

\bibitem{singer2010uniqueness}
A.~Singer and M.~Cucuringu.
\newblock Uniqueness of low-rank matrix completion by rigidity theory.
\newblock {\em SIAM Journal on Matrix Analysis and Applications}, 31(4):1621--1641, 2010.

\bibitem{sugiyama}
T.~Sugiyama and S.-i. Tanigawa.
\newblock Generic global rigidity in $\ell_p$-space and the identifiability of $p$-{C}ayley-{M}enger varieties.
\newblock {\em arXiv preprint arXiv:2402.18190}, 2024.

\bibitem{yuan2016tensor}
M.~Yuan and C.-H. Zhang.
\newblock On tensor completion via nuclear norm minimization.
\newblock {\em Foundations of Computational Mathematics}, 16(4):1031--1068, 2016.

\end{thebibliography}
\end{document}